\documentclass[11pt,reqno]{amsart}

\usepackage{lineno,hyperref}

\allowdisplaybreaks[4]

\usepackage{amsmath}
\usepackage{amssymb}
\usepackage{mathrsfs}
\usepackage{amsthm}
\usepackage{bm}
\usepackage{graphicx}
\usepackage{xcolor}  
\usepackage{tikz} 
\usetikzlibrary{arrows,shapes,chains}  
\usepackage{booktabs}
\usepackage{float}
\usepackage{subfig}
\usepackage{geometry}
\usepackage{color}

\newtheorem{Def}{Definition}[section]
\newtheorem{lem}[Def]{Lemma}
\newtheorem{theo}[Def]{Theorem}
\newtheorem{pro}[Def]{Proposition}
\newtheorem{rem}[Def]{Remark}
\newtheorem{ex}[Def]{Example}
\newtheorem{assum}{Assumption}
\newtheorem{cor}[Def]{Corollary}
\usepackage{pgf}
\definecolor{Green}{RGB}{0,128,0}

\newcommand{\LL}{\langle}
\newcommand{\RR}{\rangle}

\newcommand{\mcal}{\mathcal}

\newcommand{\mbb}{\mathbb}
\newcommand{\mbf}{\mathbf}

\newcommand{\ud}{\mathrm d}
\newcommand{\PD}{\partial}
\numberwithin{equation}{section}
\allowdisplaybreaks \allowdisplaybreaks[4]

\begin{document}

\title[Asymptotic error distribution for SPDE]{Asymptotic error distribution of accelerated exponential Euler method for parabolic  SPDEs}

\author{Jialin Hong}
\address{LSEC, ICMSEC, Academy of Mathematics and Systems Science, Chinese Academy of Sciences,    Beijing, 100190, China;
	School of Mathematical Sciences, University of Chinese Academy of Sciences, Beijing, 100049, China}
\email{hjl@lsec.cc.ac.cn}

\author{Diancong Jin}
\address{School of Mathematics and Statistics, Huazhong University of Science and Technology, Wuhan 430074, China;
	Hubei Key Laboratory of Engineering Modeling and Scientific Computing, Huazhong University of Science and Technology, Wuhan 430074, China}
\email{jindc@hust.edu.cn (Corresponding author)}

\author{Xu Wang}
\address{LSEC, ICMSEC, Academy of Mathematics and Systems Science, Chinese Academy of Sciences,    Beijing, 100190, China;
	School of Mathematical Sciences, University of Chinese Academy of Sciences, Beijing, 100049, China}
\email{wangxu@lsec.cc.ac.cn}

\author{Guanlin Yang}
\address{LSEC, ICMSEC, Academy of Mathematics and Systems Science, Chinese Academy of Sciences, Beijing, 100190, China;
	School of Mathematical Sciences, University of Chinese Academy of Sciences, Beijing, 100049, China}
\email{yangguanlin@lsec.cc.ac.cn}

\thanks{This work is supported by the National key R\&D Program of China (No.\ 2020YFA0713701), the National Natural Science Foundation of China (No. 12201228 and No. 12288201), the Fundamental Research Funds for the Central Universities (No. 3034011102), and the CAS Project for Young Scientists in Basic Research (YSBR-087).}

\keywords{stochastic partial differential equation, accelerated exponential Euler method, normalized error process, asymptotic error distribution}

\begin{abstract}
The asymptotic error distribution of numerical methods applied to stochastic ordinary differential equations has been well studied, which characterizes the evolution pattern of the error distribution in the small step-size regime. 
It is still open for stochastic partial differential equations whether the normalized error process of numerical methods admits a nontrivial limit distribution.
We answer this question by presenting the asymptotic error distribution of the temporal accelerated exponential Euler (AEE) method when applied to parabolic  stochastic partial differential equations.  In order to overcome the difficulty caused by the infinite-dimensional setting,  we establish a uniform approximation theorem for convergence in distribution.  Based on it, we derive the limit distribution of the normalized error process of the AEE method by studying the limit distribution of its certain appropriate finite-dimensional approximation process. As applications of our main result, the asymptotic error distribution of a fully discrete AEE method for the original equation and that of the AEE method for a stochastic ordinary differential equation are also obtained.
\end{abstract}

\maketitle

\textit{MSC 2020 subject classifications}:  
60B12, 60F17, 60H15, 60H35

\section{Introduction}

Concerning the numerical study for stochastic differential equations, the error analysis is crucial for assessing the accuracy and effectiveness of numerical methods. 
Most existing research focuses on the convergence analysis, particularly the strong or weak convergence
(cf. \cite{CHJ2023,HS2022,HW2019,Kloeden1992,Milstein2004,Zhang2017}).
Recognizing that the errors produced by stochastic numerical methods are stochastic processes, it is also necessary to conduct a thorough analysis of their probabilistic characteristics such as probabilistic limit theorems in the small step-size regime (cf.\ \cite{HLS2023}). This comprehensive examination is vital for a deeper insight into the behavior and performance of numerical methods within the stochastic framework.

The asymptotic error distribution of numerical methods, which is the limit distribution of the normalized error process as the step-size tends to zero, can be viewed as a kind of generalized central limit theorem. It characterizes the distribution pattern of the error process in the small step-size regime and provides the optimal convergence rate for stochastic numerical methods.
It was pioneered by \cite{Protter1991}, which gave a sufficient condition for the convergence in distribution of the error process of the Euler method when applied to globally Lipschitz stochastic ordinary differential equations (SODEs) driven by semimartingales, and was further studied by \cite{Protter1998AOP}.
For locally Lipschitz SODEs driven by Brownian motions, \cite{Protter2020SPA} provided the asymptotic error distribution of the Euler method. 
For stochastic integral equations driven by Brownian motions,  \cite{Fukasawa2023} and \cite{Nualart2023} gave the asymptotic error distribution of Euler-type methods. We also refer to \cite{HuAAP,Neuenkirch,HuBIT} for  asymptotic error distribution of numerical methods for SODEs driven by fractional Brownian motions. 
As is seen from the aforementioned references, the asymptotic error distribution of numerical methods for SODEs has been well studied. However, to the best of our knowledge, it remains unknown whether the normalized error process of a numerical method applied to stochastic partial differential equations (SPDEs) has a nonzero limit distribution. 

In this paper, we are devoted to filling the gap by studying the asymptotic error distribution of the temporal accelerated exponential Euler (AEE) method for parabolic SPDEs. 
Our main contribution is two-fold: (1) A uniform approximation theorem for convergence in distribution is established, which enables us to approximate the infinite-dimensional error process by its appropriate finite-dimensional approximation. (2) The asymptotic error distribution of the proposed method for parabolic SPDEs is presented based on the uniform approximation theorem. 

More precisely, we consider the following parabolic SPDE driven by a $Q$-Wiener process
  \begin{align}\label{SPDE}
  	\begin{cases}
  		\ud X(t)=AX(t)\ud t+F(X(t))\ud t+\ud W^Q(t), ~t\in(0,T],\\
  		X(0)=X_0\in H.
  	\end{cases}
 \end{align}
Here, $A$ is a negative definite operator which generates a strongly continuous semigroup $\{E(t)=e^{tA}\}_{t\ge 0}$ on the Hilbert space $H$, $F:H\to H$ is a nonlinear map, and $W^Q$ is an $H$-valued $Q$-Wiener process. Our main assumptions on $X_0$, $A$, $Q$ and $F$  are given in Assumptions \ref{assum1} and \ref{assum2}, under which \eqref{SPDE} admits a unique mild solution (cf.\ \cite{WangQi2015}) given by 
\begin{align}\label{mild}
	X(t)=E(t)X_0+\int_0^t E(t-s)F(X(s))\ud s+\int_0^t E(t-s)\ud W^Q(s),~ t\in[0,T].
\end{align}
The temporal AEE method was first proposed by \cite{Jentzen}:
\[
\overline{X}^m_{k+1}=e^{\tau A}\overline{X}^m_{k}+A^{-1}\left(E(\tau)-I\right)F(\overline{X}^m_{k})+\int_{t_k}^{t_{k+1}}E(t_{k+1}-s){\rm d}W^Q(s)
\] 
with the temporal step-size $\tau:=\frac{T}{m}$ for some $m\in\mbb N^+$ and $t_k=k\tau$ for $k=0,\ldots,m$.
Its time-continuous version is defined by
\begin{align}\label{AEE}
	X^m(t)=E(t)X_0+\int_0^t E(t-s)F(X^m(\kappa_m(s)))\ud s+\int_0^t E(t-s)\ud W^Q(s),~ t\in[0,T]
\end{align}
with $\kappa_m(s):=\lfloor\frac{s}{\tau}\rfloor\tau=\lfloor\frac{ms}{T}\rfloor\frac{T}{m}$ for $s\in[0,T]$, which satisfies $X^m(t_k)=\overline{X}^m_k$. Following the argument of \cite{WangQi2015}, where the convergence rate of $\{\overline{X}^m_k\}_{k\in\mbb N^+}$ is established, we show for any $t\in[0,T]$ that $X^m(t)$ converges to $X(t)$ in the mean-square sense of order $1$  for the trace class noise. To further illustrate that the first-order convergence rate is optimal and characterize the distribution pattern of the error, we define the normalized error process 
\begin{align}\label{eq:Um}
U^m(t):=m(X^m(t)-X(t)),~t\in[0,T] 
\end{align}
and study the limit distribution of $U^m(t)$ in $H$ as $m\to \infty$.

Different from the case of SODEs, the state space $H$ of \eqref{SPDE} is infinite-dimensional, which brings in some inherent difficulties.

\begin{itemize}
	\item \textit{Convergence in distribution for stochastic integrals driven by $Q$-Wiener processes.}  \\
	In the case of  SODEs driven by Brownian motions, the convergence theory in distribution for stochastic integrals driven by semimartingales plays a key role in deriving the asymptotic error distribution of numerical methods (cf.\ \cite{Fukasawa2023,Jacod1997,Protter1998AOP}). A frequently used tool is Jacod's theory of convergence in distribution of conditional  Gaussian martingales \cite{Jacod1997}, where the  limit distribution of conditional  Gaussian martingales can be provided by studying the convergence in probability of their variation processes. To the best of our knowledge, the convergence theory  in distribution of stochastic integrals driven by $Q$-Wiener processes has not been established. 
	
	\item \textit{Tightness of $\{U^m\}_{m\in\mbb N^+}$ in $\mbf C([0,T];H)$.}\\
	A basic argument in the case of SODEs for tackling the asymptotic error distribution is the tightness of the normalized error process $\{U^m\}_{m\in\mbb N^+}$ in $\mbf C([0,T];\mbb R^d)$. 
	It is usually implemented based on the Arzel\'a--Ascoli theorem and Kolmogorov continuity theorem by presenting the H\"older continuous estimate for $U^m$ in ${\mbf L^p(\Omega;\mathbb R^d)}$.
	For SPDEs, the tightness needs to be justified in $\mbf C([0,T];H)$. 
	It implies that the estimate for $U^m$ in ${\mbf L^p(\Omega;\dot{H}^\rho)}$ will be required instead with $\dot{H}^\rho:=Dom((-A)^{\rho})$ due to the infinite-dimensional setting, which is challenging for $\rho>0$.
\end{itemize}

In order to circumvent these difficulties, we establish a uniform approximation theorem for convergence in distribution (Theorem \ref{uniform approximation}) and reduce the analysis of $U^m(t)$ for any fixed $t\in[0,T]$ to that of its certain finite-dimensional approximation. 
More precisely, we first introduce an auxiliary process $\widetilde{U}^m(t)$ that has the same limit distribution as  $U^m(t)$ in $H$. 
Then a good approximation $\widetilde{U}^{m,n}(t)\in H_n:=\text{span}\{e_1,e_2,\ldots,e_n\}$, which could uniformly approximates $\widetilde{U}^m(t)$ as $n$ tends to infinity in the sense of Theorem \ref{uniform approximation}(A1), is constructed.
Here, $e_1,\ldots,e_n$ are the first $n$th eigenfunctions of $A$.
By giving the limit distribution of  $\widetilde{U}^{m,n}(t)$ with respect to $m$ in $H_n$ and then with respect to $n$ in $H$, we finally obtain that both $\widetilde{U}^m(t)$ and $U^m(t)$ converge in distribution to $U(t)$ in $H$, where $U$ is the solution of a linear SPDE (see Theorem \ref{maintheorem}). It immediately follows that the mean-square convergence rate one of the temporal AEE method for \eqref{SPDE} with trace class noise is optimal. 

As  two applications of Theorem \ref{maintheorem}, we also obtain the asymptotic error distribution for a fully discrete method based on the temporal AEE method and spatial spectral Galerkin method for \eqref{SPDE}, and that of the AEE method for SODEs with additive noise for the first time. 
We remark that the argument based on the uniform approximation theorem for convergence in distribution may also apply to other SPDEs driven by $Q$-Wiener processes.

The rest of the paper is organized as follows. Section \ref{Sec2} introduces the assumptions as well as the regularity estimates for $X$ and $X^m$. Section \ref{Sec3} presents the main result and the framework of its proof.  Then we give proofs of lemmas used to validate the main result in Section \ref{Sec4}, and two applications of the main result in Section \ref{Sec5}. 
Finally, some concluding remarks are provided in Section \ref{Sec6}.

\section{Preliminaries}\label{Sec2}
In this section, we give some basic properties of the exact solution and the numerical one, some of which are known in the literature. We begin with some notations.

For given Banach spaces  $(\mcal X,\|\cdot\|_{\mcal X})$ and $(\mcal Y,\|\cdot\|_{\mcal Y})$,  let $\mcal L(\mcal X,\mcal Y)$  be the space of bounded linear operators from $\mcal X$ to $\mcal Y$,  equipped with the usual operator norm $\|\cdot\|_{\mcal L(\mcal X,\mcal Y)}$. Especially, denote $\mcal L(\mcal X):=\mcal L(\mcal X,\mcal X)$ for short. 

Denote by $\mbf C([0,T];\mcal X)$ the space of $\mcal X$-valued continuous functions defined on $[0,T]$, endowed with the norm $\|f\|_{\mbf C([0,T];\mcal X)}:=\sup\limits_{t\in[0,T]}\|f(t)\|_{\mcal X}$. For $\alpha\in(0,1]$, denote by $\mbf C^\alpha([0,T];\mcal X)$ the space of all $\alpha$-H\"older continuous functions from $[0,T]$ to $\mcal X$ equipped with the norm $\|f\|_{\mbf C^\alpha([0,T];\mcal X)}:=\|f\|_{\mbf C([0,T];\mcal X)}+[f]_{\mbf C^\alpha([0,T];\mcal X)}$, where the semi-norm $[f]_{\mbf C^\alpha([0,T];\mcal X)}:=\sup\big\{\frac{\|f(t)-f(s)\|_{\mcal X}}{|t-s|^\alpha}:t,s\in[0,T],~s\neq t\big\}$.  

Let $\big(\Omega,\mcal F,\mbf P\big)$ be a completed probability space and $\mbf E[\cdot]$ denote the expectation operator with respect to the probability measure $\mbf P$. Let $\mbf L^p(\Omega;\mcal X)$, $p\ge1$, be the space of $p$th power integrable functions $f:\Omega\to \mcal X$, endowed with the usual norm $\|f\|_{\mbf L^p(\Omega;\mcal X)}:=\big(\mbf E\|f\|^p_{\mcal X}\big)^{1/p}$.  For convenience, the range will be omitted if $\mcal X=\mbb R$, e.g., we write $\mbf L^p(\Omega):=\mbf L^p(\Omega;\mbb R)$ for short. 

Denote by $|\cdot|$ the Euclidean norm of a vector or matrix and by $\overset{d}{\Rightarrow}$ the convergence in distribution for random variables. The notation $\epsilon\ll 1$ means that $\epsilon$ is a sufficiently small positive number.
In the sequel, we use $K_{a_1,a_2,\ldots,a_l}$  to denote  some generic constant depending on  parameters $a_1,a_2,\ldots,a_l$, which may vary for each appearance. 

\subsection{Assumptions}\label{Sec2.1}
Throughout  this paper, let $(H, \LL\cdot,\cdot\RR, \|\cdot\|)$ be a separable Hilbert space. Let $\mcal L_2(H)$ stand for the space of Hilbert--Schmidt operators $\Gamma:H\to H$, equipped with the   Hilbert--Schmidt norm $\|\Gamma\|_{\mcal L_2(H)}:=\big(\sum_{i=1}^{\infty}\|\Gamma \varphi_i\|^2\big)^{1/2}$, where $\{\varphi_i\}_{i\in\mbb N^+}$ is any orthonormal basis of $H$. If $S\in\mcal L(H)$ and $T\in\mcal L_2(H)$, then $ST\in\mcal L_2(H)$ and $\|ST\|_{\mcal L_2(H)}\le \|S\|_{\mcal L(H)}\|T\|_{\mcal L_2(H)}$. Without extra statement, we always suppose that $\{W^Q(t)\}_{t\in[0,T]}$ is a cylindrical $Q$-Wiener process on $\big(\Omega,\mcal F,\mbf P\big)$ with respect to a normal filtration $\{\mcal F_t\}_{t\in[0,T]}$, where $Q\in\mcal L(H)$ is a nonnegative and symmetric  operator with finite trace. Then, $W^Q$ has the following  expansion:
$$W^Q(t)=\sum_{i=1}^\infty Q^{\frac{1}{2}}h_i\beta_i(t),~t\in[0,T],$$
where $\{h_i\}_{i\in\mbb N^+}$ is an orthonormal basis of $H$ consisting of eigenvectors of $Q$ with corresponding eigenvalues  $\{q_k\}_{k\in\mbb N^+}$$\subseteq\mbb R$, and $\{\beta_i\}_{i\in\mbb N^+}$ is a family of independent real-valued standard Brownian motions defined on $\big(\Omega,\mcal F,\{\mcal F_t\}_{t\in[0,T]},\mbf P\big)$.

Let $(-A):Dom(A)\subseteq H\to H$ be a linear, densely defined, self-adjoint and positive definite operator, which is not necessarily bounded but with compact inverse (e.g., $A$ is the Laplace operator with homogeneous Dirichlet boundary condition). In this setting, $A$ is the infinitesimal generator of a $C_0$-semigroup of contractions $E(t)=e^{tA}$, $t\in[0,\infty)$ on $H$. In addition, there exists an increasing sequence of real numbers $\{\lambda_i\}_{i\in\mbb N^+}$ and an orthonormal basis $\{e_i\}_{i\in\mbb N^+}$ of $H$ such that $-Ae_i=\lambda_ie_i$ with $0<\lambda_1\le\lambda_2\le\cdots\le\lambda_n(\to\infty)$. For any $r\in\mbb R$, define the operator $(-A)^{\frac{r}{2}}$ by $(-A)^{\frac{r}{2}}x:=\sum_{i=1}^{\infty}\lambda_i^{\frac{r}{2}}x_ie_i$ for all 
$$x\in Dom((-A)^{\frac{r}{2}}):=\Big\{x=\sum_{i=1}^\infty x_ie_i: x_i\in\mbb R,~\|x\|_r^2:=\|(-A)^{\frac{r}{2}}x\|^2=\sum_{i=1}^\infty \lambda_i^rx_i^2<\infty \Big\}.$$
Denote $\dot{H}^r:=Dom((-A)^{\frac{r}{2}})$, which is a Hilbert space equipped with the inner product $\LL u,v\RR_{\dot{H}^r}:=\LL (-A)^{\frac{r}{2}}u,(-A)^{\frac{r}{2}}v\RR$ for $u,v\in \dot{H}^r$.  Especially, it holds that $H=\dot{H}^0$. We will frequently use the following facts (cf.\ \cite[Lemma 3.2]{Kruse2012}):
\begin{align}
		\|(-A)^rE(t)\|_{\mcal L(H)}&\le K_rt^{-r}, ~t>0, ~r\ge 0,\label{semigroup1}\\
		\|(-A)^{-\rho}(E(t)-I)\|_{\mcal L(H)}&\le K_\rho t^\rho,~t>0,~\rho\in[0,1] \label{semigroup2},
\end{align}
where both $K_r$ and $K_\rho$ are independent of $t$.

Next, we give the assumptions on $X_0$, $A$, $Q$ and $F$. For convenience, we always assume that $X_0$ is nonrandom. 	
\begin{assum}\label{assum1}
	Let $X_0\in \dot{H}^\beta$ and $A$ satisfy $\|(-A)^{\frac{\beta-1}{2}}Q^{\frac{1}{2}}\|_{\mcal L_2(H)}<\infty$ for some $\beta\in(1,2]$.
\end{assum}

\begin{assum}\label{assum2}
	The deterministic mapping $F:H\to \dot{H}^{-\eta}$ is twice Fr\'echlet differentiable for some $\eta\in[1,2)$. Furthermore, there exist $\delta\in[1,2)$, $\sigma\in[0,\beta)$ and $L>0$ such that
	\begin{align}
		&\|F(u)\|\le L(1+\|u\|),\quad \|\mcal D F(u)\|_{\mcal L(H)}\le L,\quad\forall ~u\in H, \label{Fgrow}\\
		& \|\mcal D F(u)v\|_{-\delta}\le L(1+\|u\|_1)\|v\|_{-1},\quad \forall~u\in\dot{H}^1,~v\in\dot{H}^{-1}, \label{F'}\\
		&\|\mcal D^2 F(u)(w_1,w_2)\|_{-\eta}\le L\|w_1\|\|w_2\|,\quad \forall~u,w_1,w_2\in H, \label{F''bound} \\
		&\|\mcal D^2 F(u)(w_1,w_2)-\mcal D^2 F(v)(w_1,w_2)\|_{-\eta}
		\le L\|u-v\|\|w_1\|\|w_2\|_\sigma,~\forall~ u,v,w_1\in H,~w_2\in\dot{H}^\sigma,		\label{F''continuity}
	\end{align}
	where $\mcal DF$ and $\mcal D^2 F$ denote the first and second order Fr\'echlet derivatives of $F$, respectively.
\end{assum}

It is worth mentioning that Assumption \ref{assum1} and \eqref{Fgrow}--\eqref{F''bound} in Assumption \ref{assum2} are used to derive the  mean-square convergence rate of the AEE method (cf.\ \cite{WangQi2015}).  We impose an additional condition  \eqref{F''continuity} to derive the limit distribution of  $U^m(t)$, $t\in[0,T]$.
Next, we give an example of $A$ and $F$ which satisfy Assumptions \ref{assum1} and \ref{assum2}.
\begin{ex}\label{ex1}
	Let  $H=\mbf L^2((0,1)^d)$ with $d=1,2,3$, and $A=\Delta=\sum_{i=1}^d\frac{\PD^2}{\PD x_i^2}$ be the Laplace operator with the homogeneous Dirichlet boundary condition such that $Dom(A)=H^2((0,1)^d)\cap H_0^1((0,1)^d)$. Let $F: H\to H$ be the Nemytskii operator associated with the function $f:\mbb R\to\mbb R$ through $F(u)(x)=f(u(x))$ for $u\in H$ and $x\in(0,1)^d$. Assume that $\sup_{x\in\mbb R}|f^{(i)}(x)|\le C$ for  $i=1,2,3$.
	
	By choosing a sufficiently smooth $Q$ such that $\|(-A)^{\frac{\beta-1}{2}}Q^{\frac{1}{2}}\|_{\mcal L_2(H)}<\infty$ for any $\beta\in(1,2]$, e.g., $Q=(-A)^{-\rho}$ with $\rho>0$ being sufficiently large, Assumption \ref{assum1} is satisfied. According to \cite[Example 2.2]{WangQi2015}, \eqref{Fgrow}--\eqref{F''bound} in Assumption \ref{assum2} are also satisfied provided $\delta\in[1,2)\cap(\frac{d}{2},2)$ and $\eta\in(\frac{d}{2},2)$. It then remains to verify \eqref{F''continuity}. In fact, note that $\mcal D^2 F(u)(w_1,w_2)(x)=f^{(2)}(u(x))w_1(x)w_2(x)$ for any $x\in(0,1)^d$. Letting $\sigma\in(\frac{d}{2},\beta)$ and using the Sobolev embedding $\dot{H}^\rho\hookrightarrow \mbf C([0,1]^d)$ for $\rho>\frac{d}{2}$, one has
	\begin{align*}
		&\;\|\mcal D^2 F(u)(w_1,w_2)-\mcal D^2 F(v)(w_1,w_2)\|_{-\eta}\\
		\le&\;\sup_{\|h\|\le 1}|\LL f^{(2)}(u)w_1w_2-f^{(2)}(v)w_1w_2, (-A)^{-\frac{\eta}{2}}h\RR|\\
		\le &\; \sup_{\|h\|\le 1}\|(f^{(2)}(u)-f^{(2)}(v))w_1\|_{\mbf L^1((0,1)^d)}\|w_2\|_{\mbf C([0,1]^d)}\|(-A)^{-\frac{\eta}{2}}h\|_{\mbf C([0,1]^d)}\\
		\le&\; \sup_{\|h\|\le 1}\|f^{(2)}(u)-f^{(2)}(v)\|\|w_1\|\|w_2\|_\sigma\|(-A)^{-\frac{\eta}{2}}h\|_{\eta}\\
		\le &\; K\|u-v\|\|w_1\|\|w_2\|_\sigma	
		\end{align*} 
	for $u,v,w_1\in H$ and  $w_2\in\dot{H}^\sigma$ due to the uniform boundedness of $f^{(3)}$.
\end{ex}

Under the above assumptions, one is able to gain the regularity estimates for the exact solution of \eqref{SPDE}.
We begin with the estimate on the stochastic convolution using the Burkholder--Davis--Gundy (BDG) inequality and \cite[Lemma 2.3]{WangQi2015}.
\begin{lem}\label{WQ}
	Let Assumption \ref{assum1} hold. For any $p\ge 1$ and $\gamma\in[0,\beta]$,  there exits a constant $K_T>0$ such that for $0\le s\le t\le T$, 
	\begin{align*}
		\Big\|\int_s^tE(t-r)\ud W^Q(r)\Big\|_{\mbf L^p(\Omega;\dot{H}^\gamma)}\le K_T(t-s)^{\min(\frac{\beta-\gamma}{2},\frac{1}{2})}\|(-A)^{\frac{\beta-1}{2}}Q^{\frac{1}{2}}\|_{\mcal L_2(H)}.
	\end{align*}
\end{lem}

We next recall the following regularity estimates of the solution $X(t)$ which have been introduced in \cite[Theorem 2.4]{WangQi2015} and  \cite[Corollary 5.2]{Kruse2012}.
\begin{lem}\label{Xregularity}
	Under Assumptions \ref{assum1} and \ref{assum2}, for any $p\ge 1$ and $\gamma\le\beta$,  there exits a constant $K_T>0$ such that 
	\begin{align*}
		\sup_{t\in[0,T]}\|X(t)\|_{\mbf L^p(\Omega;\dot{H}^\beta)}&\le K_T(1+\|X_0\|_{\dot{H}^\beta}),\\
		\|X(t)-X(s)\|_{\mbf L^p(\Omega;\dot{H}^\gamma)}&\le K_T|t-s|^{\min(\frac{1}{2},\frac{\beta-\gamma}{2})},~ t,s\in[0,T].
	\end{align*}
\end{lem}

\subsection{Accelerated exponential Euler method}
Based on the regularity estimate of the exact solution, in this subsection, we give the convergence and regularity results of the numerical solution $X^m$ generated by the AEE method \eqref{AEE}.

Following the same arguments as in \cite{WangQi2015}, one has that $X^m(t)$ converges to $X(t)$ with mean-square order $1$.
\begin{lem}\label{converge}
Let Assumptions \ref{assum1} and \ref{assum2} hold. Then there is $K>0$ independent of $m$ such that
$$\sup_{t\in[0,T]}\|X^m(t)-X(t)\|_{\mbf L^2(\Omega;H)}\le Km^{-1}.$$
\end{lem}  

Moreover, we also show the following regularity estimates of $X^m$, which will be used in the subsequent sections.
\begin{lem}\label{Xm-regularity}
Under Assumptions \ref{assum1} and \ref{assum2}, the following estimates for $X^m$ given in \eqref{AEE} hold.
\begin{itemize}
	\item [(i)] For any $\epsilon\ll 1$, 
	$$\sup_{m\ge 1}\sup_{t\in[0,T]}\|X^m(t)\|_{\mbf L^p(\Omega;\dot{H}^{\min(\beta,2-\epsilon)})}\le K(1+\|X_0\|_{\dot{H}^{\min(\beta,2-\epsilon)}}).$$
		\item [(ii)]  For any $\epsilon\ll 1$, $\gamma\le\min(\beta,2-\epsilon)$, $p\ge 1$, and $0\le s\le t\le T $,
	$$\sup_{m\ge 1}\|X^m(t)-X^m(s)\|_{\mbf L^p(\Omega;\dot{H}^\gamma)}\le K(t-s)^{\min(\frac{1}{2},\frac{\beta-\gamma}{2},\frac{2-\epsilon-\gamma}{2})}.$$	
\end{itemize}
\end{lem}
\begin{proof}
	It follows from the Minkowski inequality, Assumption \ref{assum2} \eqref{Fgrow} and Lemma \ref{WQ} that
	\begin{align*}
		\|X^m(t)\|_{\mbf L^p(\Omega;H)}&\le \|X_0\|+L\int_0^t(1+\|X^m(\kappa_m(s))\|_{\mbf L^p(\Omega;H)})\, \ud s+	\Big\|\int_0^tE(t-r)\ud W^Q(r)\Big\|_{\mbf L^p(\Omega;H)}\\
		&\le K(1+\|X_0\|)+K\int_0^t\|X^m(\kappa_m(s))\|_{\mbf L^p(\Omega;H)}\,\ud s,
	\end{align*} 
	where the fact $\|E(t)\|_{\mcal L(H)}\le 1$ is also used. Thus, we have
	\begin{align*}
		\sup_{r\in[0,t]}\|X^m(r)\|_{\mbf L^p(\Omega;H)}\le K(1+\|X_0\|)+K\int_0^t	\sup_{r\in[0,s]}\|X^m(r)\|_{\mbf L^p(\Omega;H)}\,\ud s.
	\end{align*}
	It along with the Gronwall inequality yields
	\begin{align}\label{sec2eq1}
		\sup_{t\in[0,T]}\|X^m(t)\|_{\mbf L^p(\Omega;H)}\le K(1+\|X_0\|).
	\end{align}  
	
	Similarly, $X^m$ in $\|\cdot\|_{\beta}$-norm reads
	$$\|X^m(t)\|_\beta\le \|X_0\|_\beta+K\int_0^t\|(-A)^{\frac{\beta}{2}}E(t-s)\|_{\mcal L(H)}(1+\|X^m(\kappa_m(s))\|)\,\ud s+K\Big\|\int_0^tE(t-s)\ud W^Q(s)\Big\|_\beta,$$
	which, together with estimates \eqref{semigroup1}, \eqref{sec2eq1} and Lemma \ref{WQ}, gives 
	\begin{align}\label{sec2eq2}
	\|X^m(t)\|_{\mbf L^p(\Omega;\dot{H}^\beta)}\le& \|X_0\|_\beta+K\int_0^t(t-s)^{-\frac{\beta}{2}}(1+\|X^m(\kappa_m(s))\|_{\mbf L^p(\Omega;H)})\, \ud s+K\notag\\
	\le& \|X_0\|_\beta+K(1+\|X_0\|)\int_0^t(t-s)^{-\frac\beta2}\ud s+K.
	\end{align}
	If $\beta\in(1,2)$, it holds immediately that
	\begin{align*}
		\|X^m(t)\|_{\mbf L^p(\Omega;\dot{H}^\beta)}
		\le K(1+\|X_0\|_\beta).
	\end{align*}
	If $\beta =2$, to ensure the well-posedness of the integral in \eqref{sec2eq2}, we need to consider the estimate of $X^m$ in $\dot H^{2-\epsilon}$ for $\epsilon\ll1$ by replacing $\beta$ with $2-\epsilon$ in \eqref{sec2eq2}, and get
	$$\|X^m(t)\|_{\mbf L^p(\Omega;\dot{H}^{2-\epsilon})}\le K(1+\|X_0\|_{2-\epsilon}),$$
	which finishes the proof of (i).

	We proceed to prove (ii). Let $\epsilon\ll 1$ and $c_\epsilon=\min(\beta,2-\epsilon)$.
	For any $\gamma\in[0,c_\epsilon]$, it holds that
	\begin{align*}
		&\;\|X^m(t)-X^m(s)\|_\gamma\\
		\le&\;\|(-A)^{\frac{\gamma}{2}}(E(t-s)-I)X^m(s)\|+\int_s^t\Big\|(-A)^{\frac{\gamma}{2}}E(t-r)F(X^m(\kappa_m(r)))\Big\|\,\ud r
		+\Big\|\int_s^tE(t-r)\ud W^Q(r)\Big\|_\gamma \\
		\le & \|(-A)^{-\frac{c_\epsilon-\gamma}{2}}(E(t-s)-I)\|_{\mcal L(H)}\|X^m(s)\|_{c_\epsilon}+L\int_s^t\|(-A)^{\frac{\gamma}{2}}E(t-r)\|_{\mcal L(H)}(1+\|X^m(\kappa_m(r))\|)\,\ud r\\
		&\; +\Big\|\int_s^tE(t-r)\ud W^Q(r)\Big\|_\gamma.
	\end{align*}
Together with \eqref{semigroup1}, \eqref{semigroup2}, the estimate of $X^m$ in (i) and Lemma \ref{WQ}, we arrive at
	\begin{align*}
		\|X^m(t)-X^m(s)\|_{\mbf L^p(\Omega;\dot{H}^\gamma)}&\le K(t-s)^{\frac{c_\epsilon-\gamma}{2}}+K\int_s^t(t-r)^{-\frac{\gamma}{2}}\,\ud r+K(t-s)^{\min(\frac{\beta-\gamma}{2},\frac{1}{2})}\\
		&\le K(t-s)^{\min(\frac{1}{2},\frac{\beta-\gamma}{2},\frac{2-\epsilon-\gamma}{2},\frac{2-\gamma}{2})},
	\end{align*}
	which finishes the proof due to the fact $\frac{2-\gamma}{2}\ge \frac{\beta-\gamma}{2}$. 
\end{proof}

\begin{rem}
According to the proof of Lemma \ref{Xm-regularity}, if  assume in addition $\|F(u)\|_{\zeta}\le K(1+\|u\|_\zeta)$ for some $\zeta\in(0,1]$, then conclusions of Lemma \ref{Xm-regularity} still hold with $\epsilon=0$.
\end{rem}

\section{Main result}\label{Sec3}
In this section, we present the main result and the basic framework of its proof. 

\subsection{Statement of main result} The following theorem gives the asymptotic error distribution of the temporal AEE method \eqref{AEE} for \eqref{SPDE}.
\begin{theo}\label{maintheorem}
	Let Assumptions \ref{assum1} and \ref{assum2} hold. For any $t\in[0,T]$, the normalized error process defined in \eqref{eq:Um} satisfies $U^{m}(t)\overset{d}{\Rightarrow}U(t)$ in $H$ as $m\to\infty$. Here, $U$ solves the following linear SPDE
	\begin{align}\label{U}
		U(t)
		=&\;\int_0^tE(t-s)\mcal D F(X(s))U(s)\ud s
		-\frac{T}{2}\int_0^tE(t-s)\mcal D F(X(s))AX(s)\ud s \nonumber\\
		&\; -\frac{T}{2}\int_0^tE(t-s)\mcal D F(X(s))F(X(s))\ud s-\frac{T}{2}\int_0^tE(t-s)\mcal D F(X(s))\ud W^{Q}(s)\nonumber\\
		&\;-\frac{\sqrt{3}T}{6}\int_0^tE(t-s)\mcal D F(X(s))\ud \widetilde{W}^{Q}(s)\notag\\
		&\: -\frac{T}{4}\int_0^tE(t-s)\sum_{k=1}^\infty \mcal D^2 F(X(s))(Q^{\frac{1}{2}}h_k,Q^{\frac{1}{2}}h_k)\ud s,
	\end{align}
	where $\widetilde{W}^{Q}(t)=\sum_{k=1}^\infty Q^{\frac{1}{2}}h_k\widetilde{\beta}_k(t)$ with $\{\widetilde{\beta}_k\}_{k\in\mbb N^+}$  being a family of independent standard Brownian motions and being independent of  $\{\beta_k\}_{k\in\mbb N^+}$.
\end{theo}

\subsection{A uniform approximation theorem for convergence in distribution} 
Before investigating the asymptotic error distribution of the temporal AEE method \eqref{AEE} given in Theorem \ref{maintheorem}, a criterion for determining the limit distribution of a family of infinite-dimensional random fields will be required, by studying the limit distribution of its proper approximation process. The criterion is established in the
following theorem, which plays the most crucial roles in the proof of Theorem \ref{maintheorem}. 

\begin{theo}\label{uniform approximation}
	Let $(\mcal X,\rho)$ be a metric space with the metric $\rho(\cdot,\cdot)$ and $Z^m, Z^{m,n}, Z^{\infty,n}$, $Z^{\infty,\infty}$ with $m,n\in\mbb N^+$ be  $\mcal X$-valued random variables defined on $(\Omega,{\mcal F},{\mbf P})$. Assume the following conditions hold:
	\begin{itemize}
		\item [(A1)] For any bounded Lipschitz continuous function  $f:\mcal X\to\mbb R$,
		$$\lim_{n\to\infty}\sup_{m\ge 1}\big|{\mbf E}f(Z^m)-{\mbf E}f(Z^{m,n})\big|=0.$$
		
		\item [(A2)] There exists $n_0\in\mbb N^+$ such that for any $n\ge n_0$, $Z^{m,n}\overset{d}{\Rightarrow}Z^{\infty,n}$ in $\mcal X$ as $m\to\infty$.
		
		\item [(A3)]  $Z^{\infty,n}\overset{d}{\Rightarrow}Z^{\infty,\infty}$ in $\mcal X$ as $n\to\infty$.
	\end{itemize}
	Then it holds that $Z^{m}\overset{d}{\Rightarrow}Z^{\infty,\infty}$ in $\mcal X$ as $m\to\infty$.
\end{theo}
\begin{proof}
	For any bounded Lipschitz continuous function  $f:\mcal X\to\mbb R$, $m\ge1$ and $n\ge n_0$,
	\begin{align*}
		&\;\big|{\mbf E} f(Z^m)-{\mbf E}f(Z^{\infty,\infty})\big|\\
		\le &\; \sup_{m\ge 1}\big|{\mbf E} f(Z^m)-{\mbf E}f(Z^{m,n})\big|+\big|{\mbf E} f(Z^{m,n})-{\mbf E}f(Z^{\infty,n})\big|+\big|{\mbf E} f(Z^{\infty,n})-{\mbf E}f(Z^{\infty,\infty})\big|.
	\end{align*}
	Letting $m\to\infty$ in the above formula and using (A2), we obtain for $n\ge n_0$ that
	$$\limsup_{m\to\infty}\big|{\mbf E} f(Z^m)-{\mbf E}f(Z^{\infty,\infty})\big|\le \sup_{m\ge 1}\big|{\mbf E} f(Z^m)-{\mbf E}f(Z^{m,n})\big|+\big|{\mbf E} f(Z^{\infty,n})-{\mbf E}f(Z^{\infty,\infty})\big|.$$	
	Then letting $n\to\infty$ and using (A1) and (A3), it follows that $\limsup\limits_{m\to\infty}\big|{\mbf E} f(Z^m)-{\mbf E}f(Z^{\infty,\infty})\big|=0$, which completes the proof.
\end{proof}
\begin{rem}\label{rem1}
	A sufficient condition for Theorem \ref{uniform approximation} (A1) is that there exists $p\ge 1$ such that 
	\[\lim\limits_{n\to\infty}\sup\limits_{m\ge1}{\mbf E}\big(\rho(Z^m,Z^{m,n})\big)^p=0.\] 
\end{rem}

In the following, Theorem \ref{uniform approximation} will be referred to as the uniform approximation theorem for convergence in distribution.

\subsection{Proof of Theorem \ref{maintheorem}}
Utilizing Theorem \ref{uniform approximation}, in this subsection, we present the proof of the main result, i.e., deriving the limit distribution of $U^m(t)$, within the framework outlined below: (1) Define an auxiliary process $\widetilde{U}^m(t)$ that shares the same limit distribution as $U^m(t)$.  (2) Construct a finite-dimensional approximation process $\widetilde{U}^{m,n}(t)\in H_n=\text{span}\{e_1,\ldots,e_n\}$, whose distribution can uniformly approximate the distribution of $\widetilde{U}^m(t)$ with respect to $m$ in the sense of Theorem \ref{uniform approximation} (A1). (3) Prove that $\widetilde{U}^m(t)\overset{d}{\Rightarrow}U(t)$ in $H$ according to Theorem \ref{uniform approximation}, by studying the limit distribution of $\widetilde{U}^{m,n}(t)$ iteratively as $m\to\infty$ and $n\to\infty$. The expressions for $\widetilde U^m(t)$ and $\widetilde{U}^{m,n}(t)$ will be specified in later texts of this subsection.

To be more specific, we next introduce  a series of intermediate lemmas to facilitate the execution of the aforementioned procedure. The proofs of these lemmas will be postponed to Section \ref{Sec4}.
We begin with the decomposition of $U^m$ according to expressions of $X$ and $X^m$ and the Taylor formula:
\begin{align}\label{Umt1}
	&\;U^m(t)=m(X^m(t)-X(t))\nonumber\\
	=&\;m\int_{0}^tE(t-s)(F(X^m(s))-F(X(s)))\ud s-m\int_0^tE(t-s)\big(F(X^m(s))-F(X^m(\kappa_m(s)))\big)\ud s\nonumber\\
	=&\;\int_0^tE(t-s)\mcal D F(X(s))U^m(s)\ud s+R^1_m(t)-m\int_0^tE(t-s)\mcal D F(X^m(\kappa_m(s)))(X^m(s)-X^m(\kappa_m(s)))\ud s\nonumber\\
	&\;-m\int_0^tE(t-s)\int_0^1(1-\lambda)\mcal D^2 F(\Theta_m(\lambda,s))(X^m(s)-X^m(\kappa_m(s)),X^m(s)-X^m(\kappa_m(s)))\ud \lambda \ud s,
\end{align}
where $\Theta_m(\lambda,s):=X^m(\kappa_m(s))+\lambda(X^m(s)-X^m(\kappa_m(s)))$ and
$$R^1_m(t):=m\int_0^tE(t-s)\int_0^1(1-\lambda)\mcal D^2 F(X(s)+\lambda(X^m(s)-X(s)))(X^m(s)-X(s),X^m(s)-X(s))\ud \lambda \ud s.$$
Further, we have
\begin{align}\label{Xms-XmKappa}
	&\;X^m(s)-X^m(\kappa_m(s))\nonumber\\
	=&\;(E(s-\kappa_m(s))-I)X^m(\kappa_m(s))+\int_{\kappa_m(s)}^sE(s-r)F(X^m(\kappa_m(r)))\ud r+O_m(s)
\end{align}
with $O_m(s):=\int_{\kappa_m(s)}^sE(s-r)\ud W^Q(r)$, $s\in[0,T]$.
Plugging \eqref{Xms-XmKappa} into \eqref{Umt1} gives
\begin{align}\label{Umt2}
	U^m(t)=\int_0^tE(t-s)\mcal D F(X(s))U^m(s)\ud s-\sum_{i=1}^4B^i_m(t)+R^1_m(t)-R^2_m(t), 
\end{align}
where
\begin{align*}
	B^1_m(t):=&\;m\int_0^tE(t-s)\mcal D F(X^m(\kappa_m(s)))(E(s-\kappa_m(s))-I)X^m(\kappa_m(s))\ud s,\\
	B^2_m(t):=&\;m\int_0^tE(t-s)\mcal D F(X^m(\kappa_m(s)))\int_{\kappa_m(s)}^sE(s-r)F(X^m(\kappa_m(r)))\ud r\ud s,\\
	B^3_m(t):=&\;m\int_0^tE(t-s)\mcal D F(X^m(\kappa_m(s)))O_m(s)\ud s,\\
	B^4_m(t):=&\;\frac{m}{2}\int_0^tE(t-s)\mcal D^2 F(X^m(\kappa_m(s)))(O_m(s),O_m(s))\ud s,\\
	R^2_m(t):=&\;m\int_0^tE(t-s)\int_0^1(1-\lambda)\mcal D^2 F(\Theta_m(\lambda,s))(X^m(s)-X^m(\kappa_m(s)),X^m(s)-X^m(\kappa_m(s)))\ud \lambda \ud s\\
	&\;-\frac{m}{2}\int_0^tE(t-s)\mcal D^2 F(X^m(\kappa_m(s)))(O_m(s),O_m(s))\ud s.
\end{align*}

Next, we will show that $R^1_m(t)$ and $R^2_m(t)$ are negligible when considering the limit distribution of $U^m(t)$. This means that $U^m(t)$ has the same limit distribution as $\widetilde{U}^m(t)$ if either of them converges in distribution, where  $\widetilde{U}^m(t)$ is called the auxiliary process and defined as the solution of following equation
\begin{align}\label{Utildem}
	\widetilde{U}^m(t)=\int_0^tE(t-s)\mcal D F(X(s))\widetilde{U}^m(s)\ud s-\sum_{i=1}^4B^i_m(t).
\end{align}

\begin{lem}\label{euqiv}
	Let Assumptions \ref{assum1} and \ref{assum2} hold. Then for any $t\in[0,T]$, $\lim\limits_{m\to\infty}\mbf E\|U^m(t)-\widetilde{U}^m(t)\|=0$.
\end{lem}
Further, we construct a family of $H_n$-valued stochastic process and show that its distribution can approximate uniformly the distribution of $\widetilde{U}^m(t)$ in the sense of Theorem \ref{uniform approximation}(A1).

Define  the projection operator $P_n: \dot{H}^\gamma\to H_n$ by $P_nv=\sum_{k=1}^n\LL v,e_k\RR e_k$, $v\in\dot{H}^\gamma$, $\gamma\ge -2$.  Define $A_n\in\mcal L(H_n)$ by $A_n:=AP_n$. Then $A_n$ generates an analytic semigroup $E_n(t)=e^{tA_n}$, $t\ge 0$ in $H_n$. Further, 
 define the operator $Q_n\in\mcal L(H)$ by $Q_nv=\sum_{k=1}^n\LL v,h_k\RR Qh_k$ and 
 $W^{Q_n}(t):=\sum_{k=1}^nQ^{\frac{1}{2}}h_k\beta_k(t)$, $t\in[0,T]$. It is easy to see that $Q_n$ is a symmetric, nonnegative definite operator on $H$ with finite trace, and
 $W^{Q_n}$ is a $Q_n$-Wiener process on $H$.
Define the process $\widetilde{U}^{m,n}(t)\in H_n$, $t\in[0,T]$ as follows:  
\begin{align}\label{Utildemn}
	\widetilde{U}^{m,n}(t)=&\int_0^t E_n(t-s)P_n\mcal D F(X(s))\widetilde{U}^{m,n}(s)\ud s \nonumber\\
	&\; -m\int_0^t E_n(t-s)P_n\mcal D F(X^m(\kappa_m(s)))(E_n(s-\kappa_m(s))-P_n)P_nX^m(\kappa_m(s))\ud s\nonumber\\
	&\;-m\int_0^t E_n(t-s)P_n\mcal D F(X^m(\kappa_m(s)))\int_{\kappa_m(s)}^sE_n(s-r)P_nF(X^m(\kappa_m(s)))\ud r \ud s\nonumber\\
	&\;-m\int_0^t E_n(t-s)P_n\mcal D F(X^m(\kappa_m(s)))O_{m,n}(s)\ud s\nonumber\\
	&\;-\frac{m}{2}\int_0^t E_n(t-s)P_n\mcal D^2 F(X^m(\kappa_m(s)))(O_{m,n}(s),O_{m,n}(s))\ud s\nonumber\\
	=:&\; I_0^{m,n}(t)-\sum_{i=1}^4I_i^{m,n}(t),
\end{align}
where $O_{m,n}(s):=\int_{\kappa_m(s)}^sE_n(s-r)P_n\ud W^{Q_n}(r)$, $s\in[0,T]$.

Then we can show that $\widetilde{U}^{m,n}(t)$ and $\widetilde{U}^m(t)$ satisfy Theorem \ref{uniform approximation}(A1).

\begin{lem}\label{Utildemn-Utildem}
	Let Assumptions \ref{assum1} and \ref{assum2} hold. Then for any $t\in[0,T]$,
	$$\lim\limits_{n\to\infty}\sup_{m\ge 1}\mbf E\|\widetilde{U}^{m,n}(t)-\widetilde{U}^m(t)\|^2=0.$$
\end{lem}

The following two lemmas give the iterated limit distribution of $\widetilde{U}^{m,n}(t)$ first with respect to $m$ then with respect to $n$. 
\begin{lem}\label{Utildemn-converge}
	Let Assumptions \ref{assum1} and \ref{assum2} hold. Then for any $t\in[0,T]$ and $n\ge 1$, $\widetilde{U}^{m,n}(t)\overset{d}{\Rightarrow}\widetilde{U}^{\infty,n}(t)$ in $H_n$ (thus also in $H$) as $m\to\infty$. Here, $\widetilde{U}^{\infty,n}$ solves the following equation
	\begin{align}\label{Utildeinftyn}
		&\;\widetilde{U}^{\infty,n}(t)\nonumber\\
		=&\;\int_0^tE_n(t-s)P_n\mcal D F(X(s))\widetilde{U}^{\infty,n}(s)\ud s
		-\frac{T}{2}\int_0^tE_n(t-s)P_n\mcal D F(X(s))A_nP_nX(s)\ud s \nonumber\\
		&\; -\frac{T}{2}\int_0^tE_n(t-s)P_n\mcal D F(X(s))P_nF(X(s))\ud s-\frac{T}{2}\int_0^tE_n(t-s)P_n\mcal D F(X(s))P_n\ud W^{Q_n}(s)\nonumber\\
		&\;-\frac{\sqrt{3}T}{6}\int_0^tE_n(t-s)P_n\mcal D F(X(s))P_n\ud \widetilde{W}^{Q_n}(s)\nonumber\\
		&\; -\frac{T}{4}\int_0^tE_n(t-s)P_n\sum_{k=1}^n\mcal D^2 F(X(s))(P_nQ^{\frac{1}{2}}h_k,P_nQ^{\frac{1}{2}}h_k)\ud s,
	\end{align}
	where $\widetilde{W}^{Q_n}(t)=\sum_{k=1}^nQ^{\frac{1}{2}}h_k\widetilde{\beta}_k(t)$, $t\in[0,T]$ with $(\widetilde{\beta}_1,\ldots, \widetilde{\beta}_n)$  being an $n$-dimensional standard Brownian motion  independent of  $({\beta}_1,\ldots,{\beta}_n)$.
\end{lem}

\begin{lem}\label{Utildeinftyn-converge}
	Let Assumptions \ref{assum1} and \ref{assum2} hold. Then for any $t\in[0,T]$, $\widetilde{U}^{\infty,n}(t)\overset{d}{\Rightarrow}U(t)$ in $H$ as $n\to\infty$. 
\end{lem}

\begin{proof}[Proof of Theorem \ref{maintheorem}]
	It follows from Lemmas \ref{Utildemn-Utildem}--\ref{Utildeinftyn-converge}, Remark \ref{rem1} and Theorem  \ref{uniform approximation} that $\widetilde{U}^m\overset{d}{\Rightarrow}U(t)$ as $m\to\infty$. Further, Lemma \ref{euqiv} implies that $\|\widetilde{U}^m-U^m\|$ converges to $0$ in probability. Finally, the conclusion comes from  Slutzky’s theorem  (cf.\ \cite[Theorem 13.18]{Klenke}). 
\end{proof}

\section{Proofs of Lemmas \ref{euqiv}--\ref{Utildeinftyn-converge}} \label{Sec4}
This section is devoted to validating the lemmas in Section \ref{Sec3}. 
\subsection{Proof of Lemma \ref{euqiv}}
It follows from  \eqref{F''bound} that 
\begin{align*}
	\|R^1_m(t)\|\le Km\int_0^t\|(-A)^{\frac{\eta}{2}}E(t-s)\|_{\mcal L(H)}\|X^m(s)-X(s)\|^2\ud s,
\end{align*}
which along with \eqref{semigroup1} and Lemma \ref{converge} yields
\begin{align}\label{R1m}
	\mbf E\|R^1_m(t)\|\le Km\int_0^t(t-s)^{-\frac{\eta}{2}}\mbf E\|X^m(s)-X(s)\|^2\ud s\le  Km^{-1}.
\end{align}

Noting that for any $u\in H$, $\mcal D^2 F(u)(\cdot,\cdot)$ is a bilinear operator on $H\times H$,  we decompose $R^2_m(t)$ into $R^2_m(t)=R^{2,1}_m(t)+R^{2,2}_m(t)$ with
\begin{align*}
	R^{2,1}_m(t)=&\;m\int_0^tE(t-s)\int_0^1(1-\lambda)\big[\mcal D^2 F(\Theta_m(\lambda,s))(X^m(s)-X^m(\kappa_m(s)),X^m(s)-X^m(\kappa_m(s)))\\
	&-\mcal D^2 F(X^m(\kappa_m(s)))(X^m(s)-X^m(\kappa_m(s)),X^m(s)-X^m(\kappa_m(s)))\big]\ud \lambda \ud s,\\
	R^{2,2}_m(t)=&\;\frac{m}{2}\int_0^t E(t-s)\mcal D^2 F(X^m(\kappa_m(s)))\big(X^m(s)-X^m(\kappa_m(s))+O_m(s),X^m(s)-X^m(\kappa_m(s))-O_m(s)\big)\ud s.
\end{align*}
Recall that $\Theta_m(\lambda,s)=X^m(\kappa_m(s))+\lambda(X^m(s)-X^m(\kappa_m(s)))$ and $O_m(s)=\int_{\kappa_m(s)}^sE(s-r)\ud W^{Q}(r)$. By \eqref{semigroup1} and \eqref{F''continuity}, 
\begin{align*}
	\|R^{2,1}_m(t)\|\le&\; Km\int_0^t\|(-A)^{\frac{\eta}{2}}E(t-s)\|_{\mcal L(H)}\int_0^1(1-\lambda)\|\Theta_m(\lambda,s)-X^m(\kappa_m(s))\|    \\
	&\;\cdot\|X^m(s)-X^m(\kappa_m(s))\|\|X^m(s)-X^m(\kappa_m(s))\|_\sigma \ud \lambda \ud s\\
	\le &\; Km\int_0^t(t-s)^{-\frac{\eta}{2}}\|X^m(s)-X^m(\kappa_m(s))\|^2\|X^m(s)-X^m(\kappa_m(s))\|_\sigma \ud s.
\end{align*}
Applying the H\"older inequality and Lemma \ref{Xm-regularity}(ii), one has that for $\epsilon\ll 1$,
\begin{align*}
	\mbf E\|R^{2,1}_m(t)\|&\le Km\int_0^t (t-s)^{-\frac{\eta}{2}}\|X^m(s)-X^m(\kappa_m(s))\|^2_{\mbf L^4(\Omega;H)}\|X^m(s)-X^m(\kappa_m(s))\|_{\mbf L^2(\Omega;\dot{H}^\sigma)}\ud s\\
	&\le Km^{-\min(\frac{1}{2},\frac{\beta-\sigma}{2},\frac{2-\epsilon-\sigma}{2})}.
\end{align*}

We proceed to estimate $\mbf E\|R^{2,2}_m(t)\|$. From \eqref{F''bound} it follows  that
\begin{align*}
	\|R^{2,2}_m(t)\|\le&\; Km\int_0^t\|(-A)^{\frac{\eta}{2}}E(t-s)\|_{\mcal L(H)}\big(\|X^m(s)-X^m(\kappa_m(s))\|+\|O_m(s)\|\big)\\
	&\;\cdot\|X^m(s)-X^m(\kappa_m(s))-O_m(s)\|\ud s,
\end{align*}
which, combined with \eqref{semigroup1}, \eqref{Xms-XmKappa}, the H\"older inequality, the Minkowski inequality, Lemmas \ref{Xm-regularity} (ii) and \ref{WQ}, yields
\begin{align*}
	&\;\mbf E\|R^{2,2}_m(t)\|\le Km\int_0^t(t-s)^{-\frac{\eta}{2}}\big(\|X^m(s)-X^m(\kappa_m(s))\|_{\mbf L^2(\Omega;H)}+\big(\|O_m(s)\|_{\mbf L^2(\Omega;H)}\big)\\
	&\cdot \Big(\|(E(s-\kappa_m(s))-I)X^m(\kappa_m(s))\|_{\mbf L^2(\Omega;H)}+\Big\|\int_{\kappa_m(s)}^sE(s-r)F(X^m(\kappa_m(r)))\ud r\Big\|_{\mbf L^2(\Omega;H)}\Big)\ud s\\
	\le &\; Km^{\frac{1}{2}}\int_0^t(t-s)^{-\frac{\eta}{2}}\Big(\|(E(s-\kappa_m(s))-I)X^m(\kappa_m(s))\|_{\mbf L^2(\Omega;H)}\\
	&\;\qquad\quad+\Big\|\int_{\kappa_m(s)}^sE(s-r)F(X^m(\kappa_m(r)))\ud r\Big\|_{\mbf L^2(\Omega;H)}\Big)\ud s.
\end{align*}
By \eqref{semigroup2} and Lemma \ref{Xm-regularity}(i), it holds that for $\epsilon\ll 1$,
\begin{align*}
	&\;\|(E(s-\kappa_m(s))-I)X^m(\kappa_m(s))\|_{\mbf L^2(\Omega;H)}\\
	\le&\; \|(-A)^{-\min(\frac{\beta}{2},\frac{2-\epsilon}{2})}(E(s-\kappa_m(s))-I)\|_{\mcal L(H)}\|X^m(\kappa_m(s))\|_{\mbf L^2(\Omega;\dot{H}^{\min(\beta,2-\epsilon)})}
	\le Km^{-\min(\frac{\beta}{2},\frac{2-\epsilon}{2})}. 
\end{align*}
Moreover, $\|E(t)\|_{\mcal L(H)}\le 1$, \eqref{Fgrow} and Lemma \eqref{Xm-regularity}(i) give
\begin{align*}
	\Big\|\int_{\kappa_m(s)}^sE(s-r)F(X^m(\kappa_m(r)))\ud r\Big\|_{\mbf L^2(\Omega;H)}\le L\int_{\kappa_m(s)}^m(1+\|X^m(\kappa_m(r))\|_{\mbf L^2(\Omega;H)})\ud r\le Km^{-1}.
\end{align*}
Combining the above formulas, we obtain
$\mbf E\|R^{2,2}_m(t)\|\le Km^{-\min(\frac{\beta-1}{2},\frac{1-\epsilon}{2})}$.
Consequently, for any $t\in[0,T]$ and $\epsilon\ll 1$,
\begin{align}\label{R2m}
	\mbf E\|R^2_m(t)\|\le Km^{-\min(\frac{\beta-1}{2},\frac{1-\epsilon}{2},\frac{\beta-\sigma}{2},\frac{2-\epsilon-\sigma}{2})}.
\end{align}

Finally, it follows from \eqref{Umt2}, \eqref{Utildem}, \eqref{Fgrow}, \eqref{R1m} and \eqref{R2m} that
\begin{align*}
	\mbf E\|U^m(t)-\widetilde{U}^m(t)\|\le K\int_0^t\mbf E\|U^m(s)-\widetilde{U}(s)\|\ud s+Km^{-\min(\frac{\beta-1}{2},\frac{1-\epsilon}{2},\frac{\beta-\sigma}{2},\frac{2-\epsilon-\sigma}{2})},
\end{align*}
which proves Lemma \eqref{euqiv} by means of the Gronwall inequality.  \hfill $\square$

\subsection{Poof of Lemma \ref{Utildemn-Utildem}} 
In order to prove  Lemma \ref{Utildemn-Utildem}, we need the following uniform upper bound for $\|\widetilde{U}^{m,n}(t)\|_{\mbf L^2(\Omega;H)}$.

\begin{lem}\label{Utildemnbounded}
	Under Assumptions \ref{assum1} and \ref{assum2}, we have that there exist $K_{\delta,\eta,T}>0$ such that
	$$\sup_{t\in[0,T]}\sup_{m\ge 1}\sup_{n\ge 1}\|\widetilde{U}^{m,n}(t)\|_{\mbf L^2(\Omega;H)}\le K_{\delta,\eta,T}.$$
\end{lem}
\begin{proof}
	By \eqref{F'}, \eqref{semigroup1}, \eqref{semigroup2}, $E_n(t)P_nu=E(t)P_nu$, $u\in H$ and $\|P_n\|_{\mcal L(H)}\le 1$, 
	\begin{align*}
		&\;\|I_1^{m,n}(t)\|\\
		\le&\; Km\int_0^t\|(-A)^{\frac{\delta}{2}}E(t-s)\|_{\mcal L(H)}(1+\|X^m(\kappa_m(s))\|_1)\|(-A)^{-1}(E(s-\kappa_m(s))-I)\|_{\mcal L(H)}\|X^m(\kappa_m(s))\|_1\ud s\\
		\le&\; K_T\int_0^t(t-s)^{-\frac{\delta}{2}}(1+\|X^m(\kappa_m(s))\|_1^2)\ud s.
	\end{align*}	
	Then Lemma \ref{Xm-regularity}(i) yields that for any $m,n\ge 1$ and $t\in[0,T]$, 
	\begin{align*}
		\|I_1^{m,n}(t)\|_{\mbf L^2(\Omega;H)}\le K_T\int_0^t(t-s)^{-\frac{\delta}{2}}\big(1+\|X^m(\kappa_m(s))\|_{\mbf L^4(\Omega;\dot{H}^1)}^2\big)\ud s\le K_{\delta,T}.
	\end{align*}
	Combining $\|E(t)\|_{\mcal L(H)}\le 1$, \eqref{Fgrow} and  Lemma \ref{Xm-regularity}(i), we have that for $m,n\ge 1$ and $t\in[0,T]$,
	\begin{align*}
		\|I_2^{m,n}(t)\|_{\mbf L^2(\Omega;H)}\le Km\int_0^t\int_{\kappa_m(s)}^s(1+\|X^m(\kappa_m(r))\|_{\mbf L^2(\Omega;H)})\ud r\ud s\le K_T.
	\end{align*}
	Further, denoting $t_k=k\tau$, $k=0,1,\ldots,m$, we write $I_3^{m,n}(t)$ as
	\begin{align*}
		I_3^{m,n}(t)=&m\sum_{k=0}^{\lfloor\frac{t}{\tau}\rfloor}\int_{t_k}^{t_{k+1}\wedge t}E_n(t-s)P_n\mcal D F(X^m(t_k))\int_{t_k}^sE_n(s-r)P_n\ud W^{Q_n}(r)\ud s\\
		=:&m\sum_{k=0}^{\lfloor\frac{t}{\tau}\rfloor}J_{k,t}.
	\end{align*}
	Since $X^m(t_k)$ is $\mcal F_{t_k}$-measurable and $\int_{t_k}^sE_n(s-r)\ud W^{Q_n}(r)$ is independent of  $\mcal F_{t_k}$, it follows from the property of conditional expectation that the  $\mbf E\LL J_{k,t},J_{l,t}\RR=0$ for $k\neq l$.
	By the BDG inequality, for $p>0$,
	\begin{align}\label{Omnbound}
		\mbf E\|O_{m,n}(s)\|^p\le K(p)\Big(\int_{\kappa_m(s)}^s\|E_n(s-r)P_nQ_n^{\frac{1}{2}}\|_{\mcal L_2(H)}^2\ud r\Big)^{\frac{p}{2}}\le K_{p,T}m^{-\frac{p}{2}}\|(-A)^{\frac{\beta-1}{2}}Q^{\frac{1}{2}}\|_{\mcal L_2(H)}^p,
	\end{align}
	where we used the fact
	\begin{align}\label{trQfinite}
		\|Q_n^{\frac{1}{2}}\|_{\mcal L_2(H)}\le \|Q^{\frac{1}{2}}\|_{\mcal L_2(H)}\le \|(-A)^{-\frac{\beta-1}{2}}\|_{\mcal L(H)}\|(-A)^{\frac{\beta-1}{2}}Q^{\frac{1}{2}}\|_{\mcal L_2(H)}<\infty.
	\end{align} 
	Therefore, the H\"older inequality, \eqref{Fgrow} and \eqref{Omnbound} lead to
	\begin{align*}
		\mbf E\|I^{m,n}_3(t)\|^2=&\;m^2\sum_{k=0}^{\lfloor\frac{t}{\tau}\rfloor}\mbf E\|J_{k,t}\|^2\le Km^2\sum_{k=0}^{\lfloor\frac{t}{\tau}\rfloor}\int_{t_k}^{t_{k+1}\wedge t}(t_{k+1}\wedge t-t_k)\mbf E\|O_{m,n}(s)\|^2\ud r\ud s\le K_T.
	\end{align*}
	It follows from \eqref{semigroup1}, \eqref{F''bound} and \eqref{Omnbound} that
	\begin{align*}
		\|I_4^{m,n}(t)\|_{\mbf L^2(\Omega;H)}\le Km\int_0^t\|(-A)^{\frac{\eta}{2}}E(t-s)\|_{\mcal L(H)}\|O_{m,n}(s)\|_{\mbf L^4(\Omega;H)}^2\ud s\le K_{\eta,T}.
	\end{align*}
	
	Combining  \eqref{Fgrow}, \eqref{Utildemn} and the above estimates for $\|I^{m,n}_i(t)\|_{\mbf L^2(\Omega;H)}$ for $i=1,2,3,4$, we conclude that  for any $m,n\ge1$ and $t\in[0,T]$,
	\begin{align*}
		\|\widetilde{U}^{m,n}(t)\|_{\mbf L^2(\Omega;H)}\le K\int_0^t	 \|\widetilde{U}^{m,n}(s)\|_{\mbf L^2(\Omega;H)}\ud s+K_{\delta,\eta,T}.
	\end{align*}
	This formula together with the Gronwall inequality finishes the proof.
\end{proof}

\vspace{5mm}
\textit{Proof of Lemma \ref{Utildemn-Utildem}.}
By \eqref{Utildemn}, \eqref{Utildem} and $E(t)P_n=E_n(t)P_n$, 
\begin{align}\label{sec3eq1}
	\widetilde{U}^{m,n}(t)-\widetilde{U}^m(t)=\int_0^tE(t-s)\mcal D F(X(s))(\widetilde{U}^{m,n}(s)-\widetilde{U}^m(s))\ud s+\sum_{i=0}^4S_i^{m,n}(t),~t\in[0,T],
\end{align}
where
\begin{align*}
	S_0^{m,n}(t):=&\;\int_0^tE(t-s)(P_n-I)\mcal D F(X(s))\widetilde{U}^{m,n}(s)\ud s,\\
	S_1^{m,n}(t):=&\;m\int_0^tE(t-s)(I-P_n)\mcal D F(X^m(\kappa_m(s)))(E(s-\kappa_m(s))-I)X^m(\kappa_m(s))\ud s\\
	&\;+m\int_0^tE(t-s)P_n\mcal D F(X^m(\kappa_m(s)))(I-P_n)(E(s-\kappa_m(s))-I)X^m(\kappa_m(s))\ud s,\\
	S_2^{m,n}(t):=&\; m\int_0^tE(t-s)(I-P_n)\mcal D F(X^m(\kappa_m(s)))\int_{\kappa_m(s)}^sE(s-r)F(X^m(\kappa_m(r)))\ud r\ud s\\
	&\;+m\int_0^tE(t-s)P_n\mcal D F(X^m(\kappa_m(s)))\int_{\kappa_m(s)}^s(I-P_n)E(s-r)F(X^m(\kappa_m(r)))\ud r\ud s,\\
	S_3^{m,n}(r)=&\;S_{3,1}^{m,n}(t)+S_{3,2}^{m,n}(t)\\
	:=&\;m\int_0^tE(t-s)(I-P_n)\mcal D F(X^m(\kappa_m(s)))O_m(s)\ud s\\
	&\;+m\int_0^tE(t-s)P_n\mcal D F(X^m(\kappa_m(s)))(O_m(s)-O_{m,n}(s))\ud s,\\
	S_4^{m,n}(t):=&\;\frac{m}{2}\int_0^tE(t-s)(I-P_n)\mcal D^2 F(X^m(\kappa_m(s)))(O_m(s),O_m(s))\ud s\\
	&\;+\frac{m}{2}\int_0^tE(t-s)P_n\mcal D^2 F(X^m(\kappa_m(s)))(O_m(s)+O_{m,n}(s),O_m(s)-O_{m,n}(s))\ud s.
\end{align*}

It follows from \eqref{semigroup1}--\eqref{Fgrow} and Lemma \ref{Utildemnbounded} that for any $m,n\ge1$, $t\in[0,T]$ and $\rho\in(0,1)$,
\begin{align*}
	\|S_0^{m,n}(t)\|_{\mbf L^2(\Omega;H)}\le L\int_0^t\|(-A)^{-\rho}(I-P_n)\|_{\mcal L(H)}\|(-A)^\rho E(t-s)\|_{\mcal L(H)}\|\widetilde{U}^{m,n}(s)\|_{\mbf L^2(\Omega;H)}\ud s \le K\lambda_{n+1}^{-\rho},
\end{align*}
where we used the fact $\|(-A)^{-\gamma}(I-P_n)\|_{\mcal L(H)}\le \lambda_{n+1}^{-\gamma}$ for $\gamma\ge 0$.
Let $\epsilon\ll 1$ and $\rho=\min(1-\frac{\delta}{2}-\epsilon,\frac{\beta-1}{2},\frac{1-\epsilon}{2})$. Then \eqref{semigroup1}, \eqref{semigroup2} and \eqref{F'} produce
\begin{align*}
	&\;\|S_1^{m,n}(t)\|\\
	\le&\; Lm\int_0^t\|(-A)^{-\rho}(I-P_n)\|_{\mcal L(H)}\|(-A)^{\rho+\frac{\delta}{2}} E(t-s)\|_{\mcal L(H)}(1+\|X^m(\kappa_m(s))\|_1)\\
	&\;\cdot\|(-A)^{-1}(E(s-\kappa_m(s))-I)\|_{\mcal L(H)}\|X^m(\kappa_m(s))\|_1\ud s+Lm\int_0^t\|(-A)^{\frac{\delta}{2}}E(t-s)\|_{\mcal L(H)}\\
	&\;\cdot(1+\|X^m(\kappa_m(s))\|_1)\|(-A)^{-\rho}(I-P_n)\|_{\mcal L(H)}\|(-A)^{-1}(E(s-\kappa_m(s))-I)\|_{\mcal L(H)}\|X^m(\kappa_m(s))\|_{1+2\rho}\ud s\\
	\le &\; K\lambda_{n+1}^{-\rho}\int_0^t(t-s)^{-(\rho+\frac{\delta}{2})}(1+\|X^m(\kappa_m(s))\|_1^2)\ud s\\
	&\;+K\lambda_{n+1}^{-\rho}\int_0^t(t-s)^{-\frac{\delta}{2}}(1+\|X^m(\kappa_m(s))\|_1)\|X^m(\kappa_m(s))\|_{\min(\beta,2-\epsilon)}\ud s.
\end{align*}
The above formula combined with the H\"older inequality and Lemma \ref{Xm-regularity}(i) yields
\begin{align*}
	\sup_{m\ge 1}\sup_{t\in[0,T]}\|S_1^{m,n}(t)\|_{\mbf L^2(\Omega\;H)}\le K\lambda_{n+1}^{-\rho}=K\lambda_{n+1}^{-\min(1-\frac{\delta}{2}-\epsilon,\frac{\beta-1}{2})}
\end{align*}
due to $1-\frac{\delta}{2}-\epsilon\le \frac{1-\epsilon}{2}$. 

Using \eqref{semigroup1}, \eqref{Fgrow} and \eqref{F'} gives
\begin{align*}
	&\;\|S_2^{m,n}(t)\|\\
	\le&\; L^2m\int_0^t\|(-A)^{-\frac{1}{2}}(I-P_n)\|_{\mcal L(H)}\|(-A)^{\frac{1}{2}}E(t-s)\|_{\mcal L(H)}\int_{\kappa_m(s)}^s(1+\|X^m(\kappa_m(r))\|)\ud r\ud s\\
	&\;+L^2m\int_0^t\|(-A)^{\frac{\delta}{2}}E(t-s)\|_{\mcal L(H)}(1+\|X^m(\kappa_m(s))\|_1)\int_{\kappa_m(s)}^s\|(-A)^{-\frac{1}{2}}(I-P_n)\|_{\mcal L(H)}(1+\|X^m(\kappa_m(r))\|)\ud r\ud s\\
	\le &\;Km\lambda_{n+1}^{-\frac{1}{2}}\int_0^t(t-s)^{-\frac{1}{2}}\int_{\kappa_m(s)}^s(1+\|X^m(\kappa_m(r))\|)\ud r\ud s\\
	&\;+Km\lambda_{n+1}^{-\frac{1}{2}}\int_0^t(t-s)^{-\frac{\delta}{2}}(1+\|X^m(\kappa_m(s))\|_1)\int_{\kappa_m(s)}^s(1+\|X^m(\kappa_m(r))\|)\ud r\ud s,
\end{align*}
which along with Lemma \ref{Xm-regularity} (i) and the H\"older inequality implies 
\[
\sup\limits_{m\ge 1}\sup\limits_{t\in[0,T]}\|S_2^{m,n}(t)\|_{\mbf L^2(\Omega;H)}\le K\lambda_{n+1}^{-\frac{1}{2}}.
\]

Similar to the treatment of $\mbf E\|I_3^{m,n}(t)\|^2$ in the proof of Lemma \ref{Utildemnbounded}, we have 
\begin{align*}
	\mbf E\|S_{3,1}^{m,n}(t)\|^2=m^2\sum_{k=0}^{\lfloor\frac{t}{\tau}\rfloor}\mbf E\Big\|\int_{t_k}^{t_{k+1}\wedge t}E(t-s)(I-P_n)\mcal D F(X^m(t_k))\int_{t_k}^sE(s-r)\ud W^Q(r)\ud s\Big\|^2.
\end{align*}
By the H\"older inequality, \eqref{semigroup1}, \eqref{Fgrow} and Lemma \ref{WQ}, we have that for $\rho\in(0,\frac{1}{2})$,
\begin{align*}
	\mbf E\|S_{3,1}^{m,n}(t)\|^2\le&\; K m^2\sum_{k=0}^{\lfloor\frac{t}{\tau}\rfloor}(t_{k+1}\wedge t-t_k)\int_{t_k}^{t_{k+1}\wedge t}\|(-A)^{-\rho}(I-P_n)\|_{\mcal L(H)}^2\|(-A)^\rho E(t-s)\|_{\mcal L(H)}^2\\
	&\;\cdot\mbf E\|\int_{t_k}^sE(s-r)\ud W^Q(r)\|^2\ud s\\
	\le &\; Km\lambda_{n+1}^{-2\rho}\sum_{k=0}^{\lfloor\frac{t}{\tau}\rfloor}\int_{t_k}^{t_{k+1}\wedge t}(t-s)^{-2\rho}(s-t_k)\ud s\le K(\rho,T)\lambda_{n+1}^{-2\rho}.
\end{align*}
We claim that for any $p\ge 1$, 
\begin{align}\label{Om-Omn}
	\|O_m(s)-O_{m,n}(s)\|_{\mbf L^p(\Omega;H)}\le Km^{-\frac{1}{2}}\Big(\lambda_{n+1}^{-\frac{\beta-1}{2}}+\big(\sum_{k=n+1}^\infty\|Q^{\frac{1}{2}}h_k\|^2\big)^{\frac{1}{2}}\Big).
\end{align}
In fact, $O_m(s)-O_{m,n}(s)=\int_{\kappa_m(s)}^s(I-P_n)E(s-r)\ud W^Q(r)+\int_{\kappa_m(s)}^sE(s-r)P_n\ud W^{Q-Q_n}(r)$ with $W^{Q-Q_n}$ being a $(Q-Q_n)$-Wiener process. Applying the BDG inequality and Lemma \ref{WQ}, we have that for $p\ge1$,
\begin{align*}
	\mbf E\|O_m(s)-O_{m,n}(s)\|^p\le&\; K_p\|(-A)^{-\frac{\beta-1}{2}}(I-P_n)\|_{\mcal L(H)}^p\mbf E\Big\|\int_{\kappa_m(s)}^sE(s-r)\ud W^Q(r)\Big\|^p_{\beta-1}\\
	&\;+K_p\Big(\int_{\kappa_m(s)}^s\|E(s-r)P_n(Q-Q_n)^{\frac{1}{2}}\|_{\mcal L_2(H)}^2\ud r\Big)^\frac{p}{2}\\
	\le &\; Km^{-\frac{p}{2}}\Big(\lambda_{n+1}^{-\frac{p(\beta-1)}{2}}+\big(\sum_{k=n+1}^\infty\|Q^{\frac{1}{2}}h_k\|^2\big)^{\frac{p}{2}}\Big),
\end{align*}
which implies the claim \eqref{Om-Omn}. Similar to the estimate for $\mbf E\|S_{3,1}^{m,n}\|^2$, it holds that
\begin{align*}
	\mbf E\|S_{3,2}^{m,n}(t)\|^2=m^2\sum_{k=0}^{\lfloor\frac{t}{\tau}\rfloor}\mbf E\Big\|\int_{t_k}^{t_{k+1}\wedge t}E(t-s)P_n\mcal D F(X^m(t_k))(O_m(s)-O_{m,n}(s))\ud s\Big\|^2.
\end{align*}
Accordingly, the H\"older inequality, \eqref{Fgrow} and  \eqref{Om-Omn} yield
\begin{align*}
	\mbf E\|S_{3,2}^{m,n}(t)\|^2\le K \Big(\lambda_{n+1}^{-(\beta-1)}+\sum_{k=n+1}^\infty\|Q^{\frac{1}{2}}h_k\|^2\Big).
\end{align*}
Noting that $\lim\limits_{n\to\infty}\sum\limits_{k=n+1}^\infty\|Q^{\frac{1}{2}}h_k\|^2=0$ due to \eqref{trQfinite}, we have
\begin{align*}
	\lim_{n\to\infty}\sup_{m\ge1}\sup_{t\in[0,T]}\|S_3^{m,n}(t)\|_{\mbf L^2(\Omega;H)}=0.
\end{align*}

Next, let us tackle  $S_4^{m,n}$. From \eqref{F''bound} we derive 
\begin{align*}
	\|S_4^{m,n}\|\le&\; Km\int_0^t\|(-A)^{-\rho}(I-P_n)\|_{\mcal L(H)}\|(-A)^{\rho+\frac{\eta}{2}}E(t-s)\|_{\mcal L(H)}\|O_m(s)\|^2\ud s\\
	&\;+Km\int_0^t\|(-A)^{\frac{\eta}{2}}E(t-s)\|_{\mcal L(H)}(\|O_m(s)\|+\|O_{m,n}(s)\|)\|O_{m}(s)-O_{m,n}(s)\|\ud s.
\end{align*}
The H\"older inequality, \eqref{semigroup1}, Lemma \ref{WQ}, \eqref{Omnbound} and \eqref{Om-Omn} give  that for $\rho\in(0,1-\frac{\eta}{2})$,
\begin{align*}
	\|S_4^{m,n}\|_{\mbf L^2(\Omega;H)}\le&\; Km \lambda_{n+1}^{-\rho}\int_0^t(t-s)^{-(\rho+\frac{\eta}{2})}\|O_m(s)\|_{\mbf L^4(\Omega;H)}^2\ud s\\
	&\; +km\int_0^t(t-s)^{-\frac{\eta}{2}}(\|O_m(s)\|_{\mbf L^4(\Omega;H)}+\|O_{m,n}(s)\|_{\mbf L^4(\Omega;H)})\|O_m(s)-O_{m,n}(s)\|_{\mbf L^4(\Omega;H)}\ud s\\
	\le&\; K\Big(\lambda_{n+1}^{-\rho}+\lambda_{n+1}^{-\frac{\beta-1}{2}}+\big(\sum_{k=n+1}^\infty\|Q^{\frac{1}{2}}h_k\|^2\big)^{\frac{1}{2}}\Big).
\end{align*}
In this way, one has  $\sup\limits_{m\ge 1}\sup\limits_{t\in[0,T]}\|S_4^{m,n}(t)\|_{\mbf L^2(\Omega;H)}\le K\lambda_{n+1}^{-\frac{1}{2}}$.

By \eqref{sec3eq1} and \eqref{Fgrow}, for any $m,n\ge1$,
\begin{align*}
	\|\widetilde{U}^{m,n}(t)-\widetilde{U}^m(t)\|_{\mbf L^2(\Omega;H)}\le K\int_0^t\|\widetilde{U}^{m,n}(s)-\widetilde{U}^m(s)\|_{\mbf L^2(\Omega;H)}\ud s+\sum_{i=0}^4\sup_{m\ge 1}\sup_{t\in[0,T]}\|S_i^{m,n}(t)\|_{\mbf L^2(\Omega;H)}.
\end{align*}
Then the Gronwall inequality yields  $$\sup\limits_{m\ge1}\sup_{t\in[0,T]}\|\widetilde{U}^{m,n}(t)-\widetilde{U}^m(t)\|_{\mbf L^2(\Omega;H)}\le K_T\sum_{i=0}^4\sup_{m\ge 1}\sup_{t\in[0,T]}\|S_i^{m,n}(t)\|_{\mbf L^2(\Omega;H)},$$
which combined with the previous estimates for $\|S_i^{m,n}(t)\|_{\mbf L^2(\Omega;H)}$, $i=0,\ldots,4$ completes the proof of Lemma \ref{Utildemn-Utildem}. \hfill $\square$

\subsection{Proof of Lemma \ref{Utildemn-converge}}
In this subsection, we prove Lemma \ref{Utildemn-converge}. For this end, we present the convergence of $I_i^{m,n}$ as $m\to\infty$ in Lemma \ref{Imnconverge}
and the tightness of some main terms of the error decomposition in Lemma \ref{Imntight}.

In the proof of Lemma \ref{Imnconverge}, we will use two useful propositions, where the first one is an adaptation of \cite[Lemma C.2]{Fukasawa2023}.
\begin{pro}\label{keyprop}
	Let $T_0,T$ be given positive numbers and $g\in\mbf L^q(0,1)$, $q\in(1,+\infty]$,  be nonnegative or nonpositive for almost every $t\in(0,1)$. Let $\{Y^m(t)\}_{t\in[0,T_0]}$, $m\in\mbb N^+$, and $\{Y(t)\}_{t\in[0,T_0]}$ be  $\mbb R^d$-valued  stochastic processes defined on $(\Omega,\mcal F,\mbf P)$ with $Y$ being  continuous almost surely. If $\lim\limits_{m\to\infty}\mbf E\int_0^{T_0}|Y^m(s)-Y(s)|^p\ud s=0$ with $p\ge 1$ satisfying  $\frac{1}{p}+\frac{1}{q}=1$, then 
	\begin{align*}
		\lim_{m\to\infty}\int_0^{T_0}Y^m(s)g\Big(\frac{ms}{T}-\big\lfloor\frac{ms}{T}\big\rfloor\Big)\,\ud s=\int_0^1g(r)\,\ud r\int_0^{T_0}Y(s)\ud s\quad\text{in}~\mbf L^p(\Omega;\mbb R^d).
	\end{align*}
\end{pro}
\begin{proof}
	We only prove the conclusion for the case $p,q>1$, and the proof for the case $p=1$ and $q=+\infty$ is similar. It has been shown in \cite[Eq.\ (34)]{Fukasawa2023} that $\lim\limits_{m\to+\infty}\int_0^{t}f(s')g(ms'-\lfloor ms'\rfloor)\ud s'=\int_0^1g(r)\ud r\int_0^{t}f(s')\ud s'$ for any $f\in\mbf C([0,t];\mbb R^d)$. Then the change of variable $s'=\frac{s}{T}$ yields for any $f\in\mbf C([0,T_0];\mbb R^d)$ that
	\begin{align*}
		&\;\lim_{m\to\infty}\int_0^{T_0}f(s)g\Big(\frac{ms}{T}-\big\lfloor\frac{ms}{T}\big\rfloor\Big)\ud s=	\lim_{m\to\infty}T\int_0^{\frac{T_0}{T}}f(s'T)g(ms'-\lfloor ms'\rfloor)\ud s'\\
		=&\;T\int_0^1g(r)\ud r\int_0^{\frac{T_0}{T}}f(s'T)\ud s'=\int_0^1g(r)\ud r\int_0^{T_0}f(s)\ud s.
	\end{align*}
	A direct computation leads to 
	\begin{align*}
		\int_0^{T_0}\Big|g\Big(\frac{ms}{T}-\big\lfloor\frac{ms}{T}\big\rfloor\Big)\Big|^q{\rm d}s=&\sum_{j=0}^{\lfloor\frac{T_0m}{T}\rfloor-1}\int_{\frac{jT}{m}}^{\frac{(j+1)T}{m}}\Big|g(\frac{ms}{T}-j)\Big|^q\ud s+\int_{\lfloor\frac{T_0m}{T}\rfloor\frac{T}{m}}^{T_0}\Big|g(\frac{ms}{T}-\lfloor\frac{T_0m}{T}\rfloor)\Big|^q\ud s\\
		=&\;\frac{T}{m}\sum_{j=0}^{\lfloor\frac{T_0m}{T}\rfloor-1}\int_0^1|g(r)|^q\ud r+\frac{T}{m}\int_{0}^{\frac{mT_0}{T}-\lfloor\frac{mT_0}{T}\rfloor}|g(r)|^q\ud r\\
		\le & \big(\frac{T}{m}\lfloor\frac{mT_0}{T}\rfloor+\frac{T}{m}\big)\int_0^1|g(r)|^q\ud r\le K_{T,T_0}<\infty.
	\end{align*}
	The remained proof can be completed following the same argument as the proof of \cite[Lemma C.1]{Fukasawa2023}.
\end{proof}

\begin{pro}\label{propA2}
	Define the operator $\mcal K_n$ by $\mcal K_n(Z)(t):=Z(t)+A_n\int_{0}^{t}E_n(t-s)Z(s)\ud s$ for $Z\in\mbf C([0,T];H_n)$. Let $\{Y(t)\}_{t\in[0,T]}$ be an $H_n$-valued continuous semimartingale with $Y(0)=0$. Then it holds almost surely that  $\int_0^tE_n(t-s)\ud Y(s)=\mcal K_n(Y)(t)$, $t\in[0,T]$.
\end{pro}
\begin{proof}
The It\^o integration by parts yields
	\begin{align*}
		\int_0^tE_n(t-s)\ud Y(s)&=\big[E_n(t-s)Y(s)\big]\big|_{s=0}^{s=t}-\int_0^tY(s)d E_n(t-s)\\
		&=Y(t)+A_n\int_0^tE_n(t-s)Y(s)\ud s=\mcal K_n(Y)(t)
	\end{align*}
	utilizing the initial condition $Y(0)=0$.
\end{proof}

\begin{lem}\label{Imnconverge}
	Let Assumptions \ref{assum1} and \ref{assum2} hold and $n\ge1$ be fixed. Then for any $t\in[0,T]$, $I_i^{m,n}(t)$ converges to $I_i^{\infty,n}(t)$ in $\mbf L^1(\Omega;H)$ as $m\to\infty$ for $i=1,2,4$. Moreover, $I_3^{m,n}\overset{d}{\Rightarrow}I_3^{\infty,n}$ in $\mbf C([0,T];H_n)$ as $m\to\infty$. Here, $I_i^{\infty,n}$, $i=1,2,3,4$ are defined as
	\begin{align*}
		I_1^{\infty,n}(t):=&\;\frac{T}{2}\int_0^tE(t-s)P_n\mcal D F(X(s))A_nX(s)\ud s, \\
		I_2^{\infty,n}(t):=&\;\frac{T}{2}\int_0^tE(t-s)P_n\mcal D F(X(s))P_nF(X(s))\ud s,\\
		I_3^{\infty,n}(t):=&\;\frac{T}{2}\int_0^tE(t-s)P_n\mcal D F(X(s))P_n\ud W^{Q_n}(s)+\frac{\sqrt{3}T}{6}\int_0^tE(t-s)P_n\mcal D F(X(s))P_n\ud \widetilde{W}^{Q_n}(s),\\
		I_4^{\infty,n}(t):=&\; \frac{T}{4}\int_0^tE(t-s)P_n\sum_{k=1}^{n}\mcal D^2 F(X(s))(P_nQ^{\frac{1}{2}}h_k,P_nQ^{\frac{1}{2}}h_k),
	\end{align*}
	where $\widetilde{W}^{Q_n}$ is the one in Lemma \ref{Utildemn-converge}.
\end{lem}
\begin{proof}
	In this proof, we always let $n\ge1$ be fixed. We divide the proof into four steps.\\
	\textbf{Step 1: Convergence of $I_1^{m,n}(t)$ in $\mbf L^1(\Omega;H_n)$.}\\
	Let $I_n$ be the identity operator on $H_n$. Using the fact $m(s-\kappa_m(s))=T\big(\frac{ms}{T}-\big\lfloor\frac{ms}{T}\big\rfloor\big)$, we have
	\begin{align*}
		I_1^{m,n}(t)=&\;m\int_0^tE(t-s)P_n\mcal D F(X^m(\kappa_m(s)))\big(e^{(s-\kappa_m(s))A_n}-I_n\big)P_nX^m(\kappa_m(s))\ud s\\
		=&\; T\int_0^tE(t-s)P_n\mcal D F(X^m(\kappa_m(s)))A_nX^m(\kappa_m(s))\big(\frac{ms}{T}-\big\lfloor\frac{ms}{T}\big\rfloor\big)\ud s+\widetilde{R}_1^{m,n}(t),
	\end{align*}
	with
	\begin{align*}
		\widetilde{R}_1^{m,n}(t)=m\int_0^tE(t-s)P_n\mcal D F(X^m(\kappa_m(s)))\big(e^{(s-\kappa_m(s))A_n}-I_n-(s-\kappa_m(s))A_n\big)P_nX^m(\kappa_m(s))\ud s.
	\end{align*}
	By the fact $\|A_n\|_{\mcal L(H_n)}=\lambda_n$ and the Taylor formula, one has
	\begin{align}\label{exptAn}
		\|e^{tA_n}-I_n-tA_n\|_{\mcal L(H_n)}\le K_nt^2,
	\end{align}
	which together with \eqref{Fgrow} and Lemma \ref{Xm-regularity}(i) leads to
	\begin{align}\label{sec3eq2}
		\mbf E\|\widetilde{R}_1^{m,n}(t)\|\le K(n)m^{-1}\int_0^t \mbf E\|X^m(\kappa_m(s))\|\ud s\le K_nm^{-1}.
	\end{align}
	
	Based on \eqref{sec3eq2} and Proposition \ref{keyprop}, it follows that $\lim\limits_{m\to\infty}\mbf E\|I_1^{m,n}(t)-I_1^{\infty,n}(t)\|=0$ once we show
	\begin{align}\label{sec3eq3}
		\lim_{m\to\infty}\mbf E\int_0^t\|E(t-s)P_n\mcal D F(X^m(\kappa_m(s)))A_nX^m(\kappa_m(s))-E(t-s)P_n\mcal D F(X(s))A_nX(s)\|\ud s=0.
	\end{align}
	It should be noted that  Proposition \ref{keyprop} also applies to $H_n$-valued stochastic processes, since $H_n$ is isometric to $\mathbb R^n$.
	Next, we justify \eqref{sec3eq3}. It follows from the Taylor formula,  $\|P_n(-A)^{\rho}\|_{\mcal L(H)}= \lambda_n^\rho$, $\rho\ge0$,  \eqref{Fgrow}, \eqref{F''bound}, the H\"older inequality, Lemmas \ref{converge} and \ref{Xm-regularity} that
	\begin{align*}
		&\;\mbf E\int_0^t\|E(t-s)P_n\mcal D F(X^m(\kappa_m(s)))A_nX^m(\kappa_m(s))-E(t-s)P_n\mcal D F(X(s))A_nX(s)\|\ud s\\
		\le &\; \mbf E\int_0^t\|E(t-s)P_n\int_0^1\mcal D^2 F\big(X(s)+\lambda(X^m(\kappa_m(s))-X(s))\big)(X^m(\kappa_m(s))-X(s),A_nX^m(\kappa_m(s)))\|\ud \lambda \ud s\\
		&\;+\mbf E\int_0^t\|E(t-s)P_n\mcal D F(X(s))A_n(X^m(\kappa_m(s))-X(s))\|\ud s\\
		\le&\; K_{n,\eta}\int_0^t\|X^m(\kappa_m(s)-X(s)\|_{\mbf L^2(\Omega;H)}\|X^m(\kappa_m(s))\|_{\mbf L^2(\Omega;H)}\ud s+K(n)\int_0^t\mbf E\|X^m(\kappa_m(s)-X(s)\| \ud s\\
		\le &\; K_{n,\eta}m^{-\frac{1}{2}}.
	\end{align*}
	Thus, we prove \eqref{sec3eq3}, which implies $\lim\limits_{m\to\infty}\mbf E\|I_1^{m,n}(t)-I_1^{\infty,n}(t)\|=0$, for any $t\in[0,T]$.
	
	\bigskip
	\textbf{Step 2: Convergence of $I_2^{m,n}(t)$ in $\mbf L^1(\Omega;H_n)$.}\\
	Note that $\int_{\kappa_m(s)}^sE_n(s-r)\ud r=A_n^{-1}(e^{(s-\kappa_m(s))A_n}-I_n)$, where $A_n^{-1}$ is the inverse of $A_n$ restricted on $H_n$. By \eqref{exptAn} and \eqref{Fgrow}, we have
	\begin{align*}
		I_2^{m,n}(t)=T\int_0^tE(t-s)P_n\mcal D F(X^m(\kappa_m(s)))P_nF(X^m(\kappa_m(s)))\big(\frac{ms}{T}-\big\lfloor\frac{ms}{T}\big\rfloor\big)\ud s+\widetilde{R}_2^{m,n}(t),
	\end{align*} 
	with $\|\widetilde{R}_2^{m,n}(t)\|\le K(n)m^{-1}\int_0^t(1+\|X^m(\kappa_m(s))\|)\ud s$. Further, Lemma \ref{Xm-regularity}(i) yields $\lim\limits_{m\to\infty}\mbf E|\widetilde{R}_2^{m,n}(t)\|=0$. 
	Following the similar argument for \eqref{sec3eq3}, we have
	\begin{align*}
		\lim_{m\to\infty}\mbf E\int_0^t\|E(t-s)P_n\mcal D F(X^m(\kappa_m(s)))P_nF(X^m(\kappa_m(s)))-E(t-s)P_n\mcal D F(X(s))P_nF(X(s))\|\ud s=0.
	\end{align*}
	Thus, we can apply Proposition \ref{keyprop} to getting
	$\lim\limits_{m\to\infty}\mbf E\|I_2^{m,n}(t)-I_2^{\infty,n}(t)\|=0$.
	
	\bigskip
	\textbf{Step 3: Convergence of $I_3^{m,n}$ in distribution.}\\
	Denote $\widetilde{V}^m(t):=m\int_0^tP_n\mcal D F(X^m(\kappa_m(s)))\int_{\kappa_m(s)}^sE_n(s-r)P_n\ud W^{Q_n}(r)\ud s$, $t\in[0,T]$. Then we rewire $I_3^{m,n}(t)$ as $I_3^{m,n}(t)=\int_0^tE_n(t-s)d \widetilde{V}^m(s)$.
	Recalling that $W^{Q_n}(t)=\sum_{k=1}^nQ^{\frac{1}{2}}h_k\beta_k(t)$ and using the stochastic Fubini theorem, we obtain
	\begin{align*}
		\widetilde{V}^m(t)&=m\sum_{k=1}^n\int_0^t\int_{\kappa_m(s)}^sP_n\mcal D F(X^m(\kappa_m(s)))e^{(s-r)A_n}P_nQ^{\frac{1}{2}}h_k\ud \beta_k(r)\ud s\\
		&=m\sum_{k=1}^n\int_0^t\int_{r}^{\kappa_m(r)+\frac{T}{m}}P_n\mcal D F(X^m(\kappa_m(s)))e^{(s-r)A_n}P_nQ^{\frac{1}{2}}h_k\,\ud s\ud \beta_k(r)\\
		&=m\sum_{k=1}^n\int_0^tP_n\mcal D F(X^m(\kappa_m(r)))A_n^{-1}\big(e^{(\kappa_m(r)+\frac{T}{m}-r)A_n}-I_n\big)P_nQ^{\frac{1}{2}}h_k\ud\beta_k(r),
	\end{align*}
	where we used the fact $\kappa_m(s)=\kappa_m(r)$ for $s\in(r,\kappa_m(r)+\frac{T}{m})$.
	
	Next, we will use \cite[Theorem 4-1]{Jacod1997} to give the convergence of $\widetilde{V}^m$ in distribution on $\mbf C([0,T];H_n)$. For this end, we define the process $V^m(t)=(V^{m,1}(t),\ldots,V^{m,n}(t))\in\mbb R^n$, $t\in[0,T]$ by  $V^{m,i}(t)=\LL \widetilde{V}^m(t),e_i\RR$, $i=1,\ldots,n$, i.e.,
	\begin{align}\label{Vmi}
		V^{m,i}(t)=m\sum_{k=1}^n\int_0^t \LL P_n\mcal D F(X^m(\kappa_m(r)))A_n^{-1}\big(e^{(\kappa_m(r)+\frac{T}{m}-r)A_n}-I_n\big)P_nQ^{\frac{1}{2}}h_k,e_i\RR\ud\beta_k(r).
	\end{align}
	Hereafter, we denote $\LL X,Y\RR_t$, $t\in[0,T]$ the  variation process between the real-valued semimartingales $\{X(t)\}_{t\in[0,T]}$ and $\{Y(t)\}_{t\in[0,T]}$. Then, it is easy to compute
	\begin{align*}
		\LL V^{m,i},V^{m,j}\RR_ t=&\;m^2\sum_{k=1}^{n}\int_0^t\big[\LL P_n\mcal D F(X^m(\kappa_m(r)))A_n^{-1}\big(e^{(\kappa_m(r)+\frac{T}{m}-r)A_n}-I_n\big)P_nQ^{\frac{1}{2}}h_k,e_i\RR\\
		&\;\cdot \LL P_n\mcal D F(X^m(\kappa_m(r)))A_n^{-1}\big(e^{(\kappa_m(r)+\frac{T}{m}-r)A_n}-I_n\big)P_nQ^{\frac{1}{2}}h_k,e_j\RR\big]\,\ud r.
	\end{align*}
	Using \eqref{exptAn} and \eqref{Fgrow}, we have
	\begin{align*}
		&\;\LL V^{m,i},V^{m,j}\RR_t\nonumber\\
		=&\;m^2\sum_{k=1}^{n}\int_0^t\big[\LL P_n\mcal D F(X^m(\kappa_m(r)))(\kappa_m(r)+\frac{T}{m}-r)P_nQ^{\frac{1}{2}}h_k,e_i\RR\nonumber\\
		&\;\cdot\LL P_n\mcal D F(X^m(\kappa_m(r)))(\kappa_m(r)+\frac{T}{m}-r)P_nQ^{\frac{1}{2}}h_k,e_j\RR\big]\ud r+\widetilde{R}_3^{m,n}(t)\nonumber\\
		=&\;T^2\sum_{k=1}^n\int_0^t\big[\LL P_n\mcal D F(X^m(\kappa_m(r)))P_nQ^{\frac{1}{2}}h_k,e_i\RR\LL P_n\mcal D F(X^m(\kappa_m(r)))P_nQ^{\frac{1}{2}}h_k,e_j\RR\big]\Big(1-\big(\frac{mr}{T}-\big\lfloor\frac{mr}{T}\big\rfloor\big)\Big)^2\,\ud r\nonumber\\
		&\;+\widetilde{R}_3^{m,n}(t)
	\end{align*}
	with $\|\widetilde{R}_3^{m,n}(t)\|\le K_{n}m^{-1}$ almost surely.
	
	It follows from the Taylor formula, \eqref{F''bound}, Lemma \ref{converge}, Lemma \ref{Xm-regularity}(ii) and $\|P_n(-A)^{\rho}\|_{\mcal L(H,H_n)}\le K_{n,\rho}$, $\rho\ge0$ that for any $k,l=1,\ldots,n$, 
	\begin{align*}
		&\;\mbf E\int_0^t\big|\LL P_n\mcal D F(X^m(\kappa_m(r)))P_nQ^{\frac{1}{2}}h_k,e_l\RR-\LL P_n\mcal D F(X(r))P_nQ^{\frac{1}{2}}h_k,e_l\RR\big|\ud r\nonumber\\
		=&\;\mbf E\int_0^t\Big|\Big\LL P_n\int_0^1\mcal D^2 F\big(X(r)+\lambda(X^m(\kappa_m(r))-X(r))\big)(X^m(\kappa_m(r))-X(r),P_nQ^{\frac{1}{2}}h_k)\ud \lambda,e_l\Big\RR \Big|\ud r\nonumber\\
		\le &\; K_n\int_0^t\|P_n(-A)^{\frac{\eta}{2}}\|_{\mcal L(H,H_n)}\mbf E\|X^m(\kappa_m(r))-X(r)\|\|P_nQ^{\frac{1}{2}}h_k
		\|\ud r\nonumber\\
		\le &\; K_{n,\eta}\int_0^t(\mbf E\|X^m(\kappa_m(r))-X^m(r)\|+\mbf E\|X^m(r)-X(r)\|)\ud r \nonumber\\
		\le &\; K_nm^{-1}.
	\end{align*}
	Combining the above formula, the H\"older inequality and the fact $|\LL P_n\mcal D F(X^m(\kappa_m(r)))P_nQ^{\frac{1}{2}}h_k,e_l\RR|\le K_n$, we arrive at
	\begin{align*}
		\mbf E\int_0^t&\Big|\LL P_n\mcal D F(X^m(\kappa_m(r)))P_nQ^{\frac{1}{2}}h_k,e_i\RR \LL P_n\mcal D F(X^m(\kappa_m(r)))P_nQ^{\frac{1}{2}}h_k,e_j\RR\\
		&-\LL P_n\mcal D F(X(r))P_nQ^{\frac{1}{2}}h_k,e_i\RR\LL P_n\mcal D F(X(r))P_nQ^{\frac{1}{2}}h_k,e_j\RR\Big|\ud r=0,~k,i,j=1,\ldots,n.
	\end{align*}
	Accordingly, we can use $\|\widetilde{R}_3^{m,n}(t)\|\le K_nm^{-1}$ and Proposition \ref{keyprop} to conclude that for any $t\in[0,T]$, 
	\begin{align}\label{sec3eq4}
		\LL V^{m,i}, V^{m,j}\RR_t\xrightarrow[\text{in}~\mbf L^1(\Omega)]{m\to\infty}\frac{T^2}{3}\sum_{k=1}^{n}\int_0^t\LL P_n\mcal D F(X(r))P_nQ^{\frac{1}{2}}h_k,e_i\RR\LL P_n\mcal D F(X(r))P_nQ^{\frac{1}{2}}h_k,e_j\RR \ud r.
	\end{align}
	Further, by \eqref{Vmi} and \eqref{exptAn},
	\begin{align*}
		\LL V^{m,i},\beta_j\RR_t=&\;m\sum_{k=1}^n\int_0^t \LL P_n\mcal D F(X^m(\kappa_m(r)))A_n^{-1}\big(e^{(\kappa_m(r)+\frac{T}{m}-r)A_n}-I_n\big)P_nQ^{\frac{1}{2}}h_j,e_i\RR\,\ud r\\
		=&\; T\int_0^t\LL P_n\mcal D F(X^m(\kappa_m(r)))P_nQ^{\frac{1}{2}}h_j,e_i\RR \big(1-\big(\frac{mr}{T}-\big\lfloor\frac{mr}{T}\big\rfloor\big)\big)\ud r+\mcal O(m^{-1}).
	\end{align*}
	Similar to the proof of  \eqref{sec3eq4},  we have
	\begin{align}\label{sec3eq5}
		\LL V^{m,i},\beta^j\RR_t\xrightarrow[\text{in}~\mbf L^1(\Omega)]{m\to\infty}\frac{T}{2}\int_0^t\LL P_n\mcal D F(X(r))P_nQ^{\frac{1}{2}}h_j,e_i\RR \ud r.
	\end{align}
	
	By \cite[Theorem 4-1]{Jacod1997}, \eqref{sec3eq4} and \eqref{sec3eq5}, one has $V^m\overset{d}{\Rightarrow} V$ in $\mbf C([0,T];\mbb R^n)$ as $m\to\infty$, where $V(t)=(V^1(t),\ldots,V^n(t))$, $t\in[0,T]$ is a $(\beta_1,\ldots,\beta_n)$-biased $\mcal F$-conditional Gaussian martingale on some extension of $(\Omega,\mcal F,\mbf P)$ and satisfies
	\begin{align}
		\LL V^i,\beta_j\RR_t&=\frac{T}{2}\int_0^t\LL P_n\mcal D F(X(r))P_nQ^{\frac{1}{2}}h_j,e_i\RR \ud r \label{sec3eq6},\\
		\LL V^i,V^j\RR_t&=\frac{T^2}{3}\sum_{k=1}^{n}\int_0^t\LL P_n\mcal D F(X(r))P_nQ^{\frac{1}{2}}h_k,e_i\RR\LL P_n\mcal D F(X(r))P_nQ^{\frac{1}{2}}h_k,e_j\RR \ud r.\label{sec3eq6'}
	\end{align}
	Then from \cite[Proposition 1-4]{Jacod1997} it follows  that $V^i$ can take the form of 
	\begin{align*}
		V^i(t)=\sum_{k=1}^n\int_0^t u^{i,k}(r)\ud\beta_k(r)+\sum_{k=1}^n\int_0^t v^{i,k}(r)\ud\widetilde{\beta}_k(r),
	\end{align*}
	where $(\widetilde{\beta}_1,\ldots,\widetilde{\beta}_n)$ is an $n$-dimensional Brownian motion independent of $(\beta_1,\ldots,\beta_n)$. Thus, it holds that
	\begin{align}\label{sec3eq7}
		\LL V^i,\beta_j\RR_t&=\int_0^t u^{i,j}(r)\ud r, ~	\LL V^i,V^j\RR_t=\sum_{k=1}^n\int_0^t u^{i,k}(r)u^{j,k}(r)\ud r+\sum_{k=1}^n\int_0^t v^{i,k}(r)v^{j,k}(r)\ud r.
	\end{align}
	Comparing \eqref{sec3eq6}--\eqref{sec3eq6'} and \eqref{sec3eq7}, we can take
	\begin{align*}
		u^{i,k}(r)=\frac{T}{2}\LL P_n\mcal D F(X(r))P_nQ^{\frac{1}{2}}h_k,e_i\RR,\quad v^{i,k}(r)=\frac{\sqrt{3}T}{6}\LL P_n\mcal D F(X(r))P_nQ^{\frac{1}{2}}h_k,e_i\RR.
	\end{align*}
	Since $\widetilde{V}^m$ is isometric to $V^m$ and $V^m\overset{d}{\Rightarrow} V$ in $\mbf C([0,T];\mbb R^n)$, we have $\widetilde{V}^m\overset{d}{\Rightarrow} \widetilde{V}$ in $\mbf C([0,T]; H_n)$ with
	\begin{align*}
		\widetilde{V}(t)&=\sum_{i=1}^n V^{i}(t)e_i
		=\frac{T}{2}\sum_{k=1}^n\int_0^t P_n\mcal D F(X(r))P_nQ^{\frac{1}{2}}h_k\ud \beta_k(r)+\frac{\sqrt{3}T}{6}\sum_{k=1}^n\int_0^t P_n\mcal D F(X(r))P_nQ^{\frac{1}{2}}h_k\ud \widetilde\beta_k(r)\\
		&=\frac{T}{2}\int_0^t P_n\mcal D F(X(r))P_n\ud W^{Q_n}(r)+\frac{\sqrt{3}T}{6}\int_0^t P_n\mcal D F(X(r))P_n\ud \widetilde{W}^{Q_n}(r).
	\end{align*}
	
	Further, Proposition \ref{propA2} yields
	\begin{align*}
		I_3^{m,n}(t)=\int_0^tE_n(t-s)d\widetilde{V}^m(s)=\mcal K_n(\widetilde{V}^m)(t).
	\end{align*}
	It is straightforward to verify that $\mcal K_n$ is continuous from $\mbf C([0,T];H_n)$ to itself. Thus, the continuous mapping theorem and $\widetilde{V}^m\overset{d}{\Rightarrow}\widetilde{V}$ in  $\mbf C([0,T];H_n)$ as $m\to\infty$ give $I_3^{m,n}\overset{d}{\Rightarrow}\mcal K_n(\widetilde{V})$ in  $\mbf C([0,T];H_n)$ as $m\to\infty$. Again applying Proposition \ref{propA2}, we have  $\mcal K_n(\widetilde{V})(t)=\int_0^t E_n(t-s)d\widetilde{V}(s)=I_3^{\infty,n}(t)$, $t\in[0,T]$. 
	
	\bigskip
	\textbf{Step 4: Convergence of $I_4 ^{m,n}(t)$ in $\mbf L^1(\Omega;H_n)$.}\\
	Recall that 
	\begin{align*}
		I_4^{m,n}(t)=\frac{m}{2}\int_0^tE(t-s)P_n\mcal D^2 F(X^m(\kappa_m(s)))\Big(\int_{\kappa_m(s)}^sE_n(s-r)P_n\ud W^{Q_n}(r),\int_{\kappa_m(s)}^sE_n(s-r)P_n\ud W^{Q_n}(r)\Big)\ud s.
	\end{align*}
	Plugging the expression of $W^{Q_n}$ and using the bilinearity of $\mcal D^2 F(X^m(\kappa_m(s)))$ yield
	\begin{align*}
		I_4^{m,n}(t)=&\;\frac{m}{2}\sum_{k=1}^n\sum_{l=1}^{n}\int_0^tE(t-s)P_n\mcal D^2 F(X^m(\kappa_m(s)))\Big(\int_{\kappa_m(s)}^sE_n(s-r)P_nQ^{\frac{1}{2}}h_k\ud \beta_k(r),\\
		&\;\qquad\qquad\int_{\kappa_m(s)}^sE_n(s-r)P_nQ^{\frac{1}{2}}h_l\ud \beta_l(r)\Big)\ud s.
	\end{align*}
	Noting that $E_n(s-r)P_nQ^{\frac{1}{2}}h_k=\sum_{i=1}^ne^{-(s-r)\lambda_i}\LL P_nQ^{\frac{1}{2}}h_k,e_i\RR e_i$, we have
	\begin{align*}
		I_4^{m,n}(t)=&\;\frac{m}{2}\sum_{k,l,i,j=1}^n\LL P_nQ^{\frac{1}{2}}h_k,e_i\RR\LL P_nQ^{\frac{1}{2}}h_l,e_j\RR\int_0^t\Big(E(t-s)P_n\mcal D^2 F(X^m(\kappa_m(s)))(e_i,e_j)e^{-(\lambda_i+\lambda_j)s}\\
		&\;\qquad\int_{\kappa_m(s)}^se^{\lambda_ir}\ud\beta_k(r)\int_{\kappa_m(s)}^se^{\lambda_jr}\ud\beta_l(r) \Big)\ud s.
	\end{align*}
	It follows from the It\^o integration by parts that for any $s\in(\kappa_m(s),\kappa_m(s)+\frac{T}{m})$, 
	\begin{align*}
		&\;\int_{\kappa_m(s)}^se^{\lambda_ir}\ud\beta_k(r)\int_{\kappa_m(s)}^se^{\lambda_jr}\ud\beta_l(r)\\
		=&\;\int_{\kappa_m(s)}^s\Big(\int_{\kappa_m(s)}^re^{\lambda_i\sigma}\ud\beta_k(\sigma)\Big)e^{\lambda_j r}\ud\beta_l(r)+\int_{\kappa_m(s)}^s\Big(\int_{\kappa_m(s)}^re^{\lambda_j\sigma}\ud\beta_l(\sigma)\Big)e^{\lambda_i r}\ud\beta_k(r)+\int_{\kappa_m(s)}^se^{(\lambda_i+\lambda_j)r}\delta_{kl}\ud r,
	\end{align*}
	where $\delta_{kl}=1$ if $k=l$ and $\delta_{kl}=0$ if $k\neq l$. Denote
	\begin{gather*}
		A_m^{k,l,i,j}(s)=mP_n\mcal D^2 F(X^m(\kappa_m(s)))(e_i,e_j)e^{-(\lambda_i+\lambda_j)s}\int_{\kappa_m(s)}^s\Big(\int_{\kappa_m(s)}^re^{\lambda_i\sigma}\ud\beta_k(\sigma)\Big)e^{\lambda_j r}\ud\beta_l(r),\\
		B_m^{k,l,i,j}(s)=mP_n\mcal D^2 F(X^m(\kappa_m(s)))(e_i,e_j)e^{-(\lambda_i+\lambda_j)s}\int_{\kappa_m(s)}^s\Big(\int_{\kappa_m(s)}^re^{\lambda_j\sigma}\ud\beta_l(\sigma)\Big)e^{\lambda_i r}\ud\beta_k(r).
	\end{gather*}
	Then we have $I_4^{m,n}(t)=\sum_{i=1}^3Z_1^{m,n}(t)$, where
	\begin{align*}
		Z_1^{m,n}(t)=&\;\frac{1}{2}\sum_{k,l,i,j=1}^n\LL P_nQ^{\frac{1}{2}}h_k,e_i\RR\LL P_nQ^{\frac{1}{2}}h_l,e_j\RR\int_0^tE_n(t-s)A_m^{k,l,i,j}(s)\ud s,\\
		Z_2^{m,n}(t)=&\;\frac{1}{2}\sum_{k,l,i,j=1}^n\LL P_nQ^{\frac{1}{2}}h_k,e_i\RR\LL P_nQ^{\frac{1}{2}}h_l,e_j\RR\int_0^tE_n(t-s)B_m^{l,k,j,i}(s)\ud s,\\
		Z_3^{m,n}(t)=&\;\frac{m}{2}\sum_{k,i,j=1}^{n}\LL P_nQ^{\frac{1}{2}}h_k,e_i\RR\LL P_nQ^{\frac{1}{2}}h_k,e_j\RR\\
		&\;\cdot\int_0^tE_n(t-s)P_n\mcal D^2 F(X^m(\kappa_m(s)))(e_i,e_j)\frac{1}{\lambda_i+\lambda_j}(1-e^{-(\lambda_i+\lambda_j)(s-\kappa_m(s))})\ud s.
	\end{align*}
	
	For any $k,l,i,j\in\{1,\ldots,n\}$, when $v<\kappa_m(s)$, it holds that
	\begin{align*}
		&\;\mbf E\Big(\big\LL E_n(t-s)A_m^{k,l,i,j}(s),E_n(t-v)A_m^{k,l,i,j}(v)\big\RR\big|\mcal F_{\kappa_m(s)}\Big)\\
		=&\; \Big\LL E_n(t-v)A_m^{k,l,i,j}(v),E_n(t-s)mP_n\mcal D^2 F(X^m(\kappa_ms)))(e_i,e_j)e^{-(\lambda_i+\lambda_j)s}\\
		&\;\cdot\mbf E\Big[\int_{\kappa_m(s)}^s\Big(\int_{\kappa_m(s)}^re^{\lambda_i\sigma}\ud\beta_k(\sigma)\Big)e^{\lambda_j r}\ud\beta_l(r)\Big]\Big\RR=0.
	\end{align*}
	Thus $\mbf E\big\LL E_n(t-s)A_m^{k,l,i,j}(s),E_n(t-v)A_m^{k,l,i,j}(v)\big\RR=0$ for $v<\kappa_m(s)$, which gives
	\begin{align}\label{sec3eq8}
		&\;\mbf E\Big\|\int_0^tE_n(t-s)A_m^{k,l,i,j}(s)\ud s\Big\|^2\nonumber\\
		=&\;\mbf E\int_0^t\int_0^t\big\LL E_n(t-s)A_m^{k,l,i,j}(s),E_n(t-v)A_m^{k,l,i,j}(v)\big\RR dv \ud s\nonumber\\
		=&\; 2\mbf E\int_0^t\int_0^s\big\LL E_n(t-s)A_m^{k,l,i,j}(s),E_n(t-v)A_m^{k,l,i,j}(v)\big\RR dv \ud s\nonumber\\
		=&\; 2\int_0^t\int_{\kappa_m(s)}^s\mbf E\big\LL E_n(t-s)A_m^{k,l,i,j}(s),E_n(t-v)A_m^{k,l,i,j}(v)\big\RR dv \ud s.
	\end{align}
	By the It\^o isometry, \eqref{F''bound}, $\|P_n(-A)^{\frac{\eta}{2}}\|_{\mcal L(H,H_n)}\le K_n$, for any $k,l,i,j\in\{1,\ldots,n\}$ and $s\in[0,T]$,
	\begin{align*}
		\mbf E\|A^{k,l,i,j}_m(s)\|^2\le Km^2e^{-2(\lambda_i+\lambda_j)s}\|P_n(-A)^{\frac{\eta}{2}}\|_{\mcal L(H,H_n)}^2\int_{\kappa_m(s)}^s\int_{\kappa_m(s)}^re^{2\lambda_i\sigma}d\sigma e^{2\lambda_j r}\ud r\le K_n.
	\end{align*}
	The above formula, \eqref{sec3eq8} and the H\"older inequality yield 
	$$\mbf E\Big\|\int_0^tE_n(t-s)A_m^{k,l,i,j}(s)\ud s\Big\|^2\le K_n\int_0^t\int_{\kappa_m(s)}^sdv\ud s\le K_nm^{-1}\to 0, ~\forall ~n\ge 1,$$
	which implies $Z_1^{m,n}(t)\xrightarrow[\text{in}~\mbf L^2(\Omega;H_n)]{m\to\infty}0$ for any $n\ge 1$ and $t\in[0,T]$, and similarly
	$Z_2^{m,n}(t)\xrightarrow[\text{in}~\mbf L^2(\Omega;H_n)]{m\to\infty}0$.
	
	Note that $\frac{1}{\lambda_i+\lambda_j}(1-e^{-(\lambda_i+\lambda_j)(s-\kappa_m(s))})=(s-\kappa_m(s))+\mcal O(m^{-2})$, which together with \eqref{F''bound} further leads to
	\begin{align*}
		&\;m\int_0^tE(t-s)P_n\mcal D^2 F(X^m(\kappa_m(s)))(e_i,e_j)\frac{1}{\lambda_i+\lambda_j}(1-e^{-(\lambda_i+\lambda_j)(s-\kappa_m(s))})\ud s\\
		=&\; T\int_0^t E(t-s)P_n\mcal D^2 F(X^m(\kappa_m(s)))(e_i,e_j)\big(\frac{ms}{T}-\big\lfloor\frac{ms}{T}\big\rfloor\big)\ud s+\mcal O(m^{-1}).
	\end{align*}
	In addition, from \eqref{F''continuity}, Lemma \ref{Xm-regularity}(i) and Lemma \ref{converge} we conclude that for any $i,j=1,\ldots,n$,
	\begin{align*}
		&\;\mbf E\int_0^t\big\|E(t-s)P_n\mcal D^2 F(X^m(\kappa_m(s)))(e_i,e_j)-E(t-s)P_n\mcal D^2 F(X(s))(e_i,e_j)\big\|\ud s\\
		\le &\; \|P_n(-A)^{\frac{\eta}{2}}\|_{\mcal L(H,H_n)}\int_0^t\mbf E\|X^m(\kappa_m(s))-X(s)\|\|e_i\|\|e_j\|_{\sigma}\ud s\le K_{n,\sigma,\eta}m^{-\frac{1}{2}}.
	\end{align*}
	Then we apply Proposition \ref{keyprop} to getting
	\begin{align*}
		&\;m\int_0^tE(t-s)P_n\mcal D^2 F(X^m(\kappa_m(s)))(e_i,e_j)\frac{1}{\lambda_i+\lambda_j}(1-e^{-(\lambda_i+\lambda_j)(s-\kappa_m(s))})\ud s\\
		&\;\xrightarrow[\text{in}~\mbf L^1(\Omega;H_n)]{m\to\infty}\frac{T}{2}\int_0^tE_n(t-s)P_n\mcal D^2 F(X(s))(e_i,e_j)\ud s.
	\end{align*}
	Combining this above relation and $Z_i^{m,n}(t)\xrightarrow[\text{in}~\mbf L^2(\Omega;H_n)]{m\to\infty}0$, $i=1,2$, we obtain that $I_4^{m,n}(t)$, in $\mbf L^1(\Omega;H_n)$,  converges to 
	\begin{align*}
		&\;\frac{T}{4}\sum_{k,i,j=1}^{n}\LL P_nQ^{\frac{1}{2}}h_k,e_i\RR\LL P_nQ^{\frac{1}{2}}h_k,e_j\RR\int_0^tE_n(t-s)P_n\mcal D^2 F(X(s))(e_i,e_j)\ud s\\
		=&\;\frac{T}{4}\int_{0}^{t}E(t-s)\sum_{k=1}^nP_n\mcal D^2 F(X(s))\big(\sum_{i=1}^{n}\LL P_nQ^{\frac{1}{2}}h_k,e_i\RR e_i,\sum_{j=1}^{n}\LL P_nQ^{\frac{1}{2}}h_k,e_j\RR e_j\big)\ud s\\
		=&\; I_4^{\infty,n}(t),
	\end{align*}  
	which completes the proof.
\end{proof}

\begin{lem}\label{Imntight}
	Let Assumptions \ref{assum1} and \ref{assum2} hold. Denote $\mcal H_n={\rm{span}}\{h_1,\ldots,h_k\}$. Then for any $n\in\mbb N^+$, $\big\{\big(\widetilde{U}^{m,n},I_0^{m,n},I_1^{m,n},I_2^{m,n},I_3^{m,n},I_4^{m,n},W^{Q_n},\widetilde{W}^{Q_n},X\big)\big\}_{m\ge1}$  is tight in $\mbf C([0,T];H_n)^{\otimes 6}\times \mbf C([0,T];\mcal H_n)^{\otimes 2}\times \mbf C([0,T];H)$.
\end{lem}
\begin{proof}
	It suffices to show that every component of $\big\{\big(\widetilde{U}^{m,n},I_0^{m,n},I_1^{m,n},I_2^{m,n},I_3^{m,n},I_4^{m,n},W^{Q_n},\widetilde{W}^{Q_n},X\big)\big\}_{m\ge1}$  is tight.	It follows from \eqref{Fgrow},  Lemma \ref{Utildemnbounded} and $\|E_n(t)-I_n\|_{\mcal L(H_n)}\le \lambda_nt$, $t\ge 0$ that
	\begin{align*}
		\|I_0^{m,n}(t)-I_0^{m,n}(s)\|_{\mbf L^2(\Omega;H_n)}\le&\; \Big\|\int_s^tE_n(t-r)P_n\mcal D F(X(r))\widetilde{U}^{m,n}(r)\ud r\Big\|_{\mbf L^2(\Omega;H_n)}\\
		&\;+\Big\|\int_0^s(E_n(t-s)-I_n)E_n(s-r)P_n\mcal D F(X(r))\widetilde{U}^{m,n}(r)\ud r\Big\|_{\mbf L^2(\Omega;H_n)}\\
		\le&\; K_n(t-s),~t>s.
	\end{align*}
	Then applying the Kolmogorov continuity theorem (cf. \cite[Theorem 2]{KolContinuity2018}) gives 
	\begin{align*}
		\sup_{m\ge 1}\bigg\|\sup_{t\neq s,t,s\in[0,T]}\frac{\|I^{m,n}_0(t)-I^{m,n}(s)\|}{|t-s|^{1/4}}\bigg\|_{\mbf L^2(\Omega)}\le K_{n,T},
	\end{align*}
	which together with $I_0^{m,n}(0)=0$ yields
	$\sup\limits_{m\ge1}\big\|\|I^{m,n}_0\|_{\mbf C([0,T];H_n)}\big\|_{\mbf L^2(\Omega)}\le K_{n,T}$. The above two formulas give
	\begin{align}\label{sec3eq9}
		\sup_{m\ge 1}\mbf E\|I^{m,n}_0\|_{\mbf C^{1/4}([0,T];H_n)}\le K_{n,T}.
	\end{align}
	Denote $C_R=\{f\in\mbf C([0,T];H_n):\|f\|_{\mbf C^{1/4}([0,T];H_n)}\le R\}$, $R>0$. 
	By the Arzel\'a--Ascoli theorem, for any $R>0$, $\bar{C}_R$, the closure of $C_R$, is a compact subset of $\mbf C([0,T];H_n)$. Using \eqref{sec3eq9} and the Markov inequality, we have that for any $m\ge 1$,
	\begin{align*}
		\mbf P(I_0^{m,n}\in (\bar{C}_R)^c)\le 	\mbf P(I_0^{m,n}\in (C_R)^c)\le  \frac{1}{R}\mbf E\|I_0^{m,n}\|_{\mbf C^{1/4}([0,T];H_n)}\le \frac{K_{n,T}}{R}\to 0~\text{as}~R\to\infty.
	\end{align*}
	This indicates that (the law of) $\{I_0^{m,n}\}_{m\ge 1}$ is tight in $\mbf C([0,T];H_n)$. Similarly, one can prove the tightness of $\{I_i^{m,n}\}_{m\ge 1}$  in $\mbf C([0,T];H_n)$  by verifying $\|I^{m,n}_i(t)-I^{m,n}_i(s)\|_{\mbf L^2(\Omega;H_n)}\le K_n|t-s|$ for $i=1,2,4$. By Lemma \ref{Imnconverge}, $I_3^{m,n}\overset{d}{\Rightarrow}I_3^{\infty,n}$ in $\mbf C([0,T];H_n)$ as $m\to\infty$, which yields the tightness of $\{I_3^{m,n}\}_{m\ge 1}$ in  $\mbf C([0,T];H_n)$ due to Prokhorov's theorem.
	
	By the BDG inequality, for any $p\ge1$, 
	\begin{align*}
		\|W^{Q_n}(t)-W^{Q_n}(s)\|_{\mbf L^p(\Omega;\mcal H_n)}+	\|\widetilde{W}^{Q_n}(t)-\widetilde{W}^{Q_n}(s)\|_{\mbf L^p(\Omega;\mcal H_n)}\le K_n|t-s|^{\frac{1}{2}}.
	\end{align*} 
	Following the argument for the tightness of $\{I_0^{m,n}\}_{m\ge1}$, we conclude that for any $n\ge1$, $W^{Q_n}$ and $\widetilde{W}^{Q_n}$ are tight in $\mbf C([0,T];\mcal H_n)$, viewed a family of constant random variables with $m$ as the parameter.
	
	Finally, we show the tightness of $X$ in $\mbf C([0,T];H)$. Let   $0<\alpha<\min(\frac{1}{2},\frac{\beta}{4})$. It follows from the Lemma \ref{Xregularity} and the Kolmogorov continuity theorem (cf.\ \cite[Theorem 2]{KolContinuity2018}) that $\mbf E[X]^p_{\mbf C^{\alpha}([0,T];\dot{H}^{\beta/2})}\le K$ for  $p\gg 1$.  This along with $\|X_0\|_{\beta}<\infty$ gives
	$\mbf E\|X\|^p_{\mbf C([0,T];\dot{H}^{\beta/2})}\le K$. In addition, we also get  $\mbf E[X]^p_{\mbf C^{\alpha}([0,T];H)}\le K$ due to $\|\cdot\|\le K_\beta \|\cdot\|_{\beta/2}$. 
	Denote $\widehat{C}_R:=\{f\in\mbf C([0,T];H):\|f\|_{\mbf C([0,T];\dot{H}^{\beta/2})}+[f]_{\mbf C^{\alpha}([0,T];H)}\le R\}$, $R>0$. By the generalized Arzel\'a--Ascoli theorem (cf.\ \cite[Theorem 47.1]{AAtheorem}) and the compact embedding $\dot{H}^{\beta/2}\hookrightarrow H$, we obtain that for any $R>0$, $\widehat{C}_R$ is a pre-compact subset of $\mbf C([0,T];H)$. Further, the Markov inequality yields
	\begin{align*}
		\mbf P(X\in \overline{\widehat{C}_R}^c)&\le \mbf P(X\in (\widehat{C}_R)^c)\le \mbf P(\|X\|_{\mbf C([0,T];\dot{H}^{\beta/2})}>R/2)+ \mbf P(\|X\|_{\mbf C^\alpha([0,T];H)}>R/2)\\
		&\le \frac{K}{R}\to 0~\text{as}~R\to\infty.
	\end{align*}
	This justifies the tightness of $X$ in $\mbf C([0,T];H)$. Thus the proof is complete. 
\end{proof}

\vspace{5mm}
\textit{Proof of Lemma \ref{Utildemn-converge}.}
 	Let $n$ be fixed in this proof. By Prokhorov's theorem and Lemma \ref{Imntight}, for any subsequence $\big\{\big(\widetilde{U}^{m',n},I_0^{m',n},I_1^{m',n},I_2^{m',n},I_3^{m',n},I_4^{m',n},W^{Q_n},\widetilde{W}^{Q_n},X\big)\big\}_{m'\ge1}$ 
	of 
	$\big\{\big(\widetilde{U}^{m,n},I_0^{m,n},I_1^{m,n},I_2^{m,n},$\\$I_3^{m,n},I_4^{m,n},W^{Q_n},\widetilde{W}^{Q_n},X\big)\big\}_{m\ge1}$, there exists a further subsequence 
 	$\big\{\big(\widetilde{U}^{m_k',n},I_0^{m_k',n},I_1^{m_k',n},I_2^{m'_k,n},I_3^{m'_k,n},$\\$I_4^{m'_k,n},W^{Q_n},\widetilde{W}^{Q_n},X\big)\big\}_{k\ge1}$ converging in distribution to some
 	$\big(\bar{U}^{\infty,n},\bar I_0^{\infty,n},\bar I_1^{\infty,n},\bar I_2^{\infty,n},\bar I_3^{\infty,n},
 	\bar I_4^{\infty,n},B^{Q_n},$\\$\widetilde{B}^{Q_n},\bar X\big)$ in $\mbf C([0,T];H_n)^{\otimes 6}\times \mbf C([0,T];\mcal H_n)^{\otimes 2}\times \mbf C([0,T];H)$, maybe in some new probability space $(\bar{\Omega},\bar{\mcal F},\bar{\mbf P})$, where $(B^{Q_n},\widetilde{B}^{Q_n},\bar X)\overset{d}{=}(W^{Q_n},\widetilde{W}^{Q_n},X)$ in $\mbf C([0,T];\mcal H_n)^{\otimes 2}\times\mbf C([0,T];H)$.
 	This combined with  $\widetilde{U}^{m_k',n}(t)-I_0^{m_k',n}(t)+\sum_{i=1}^4 I_i^{m'_k,n}(t)=0$ yields 
 	\begin{align}\label{sec3eq11}
 		\bar{U}^{\infty,n}(t)-\bar{I}_0^{\infty,n}(t)+\sum_{i=1}^4 \bar I_i^{\infty,n}(t)=0,\quad\forall~t\in[0,T].
 	\end{align}

 	Define the operator $J:\mbf C([0,T];H)\times\mbf C([0,T];H_n)\to \mbf C([0,T];H)$ by 
	\[
	J(f,g)(t):=\int_0^tE_n(t-s)P_n\mcal D F(f(s))g(s)\ud s,\quad t\in[0,T].
	\] 
	It is straightforward to verify that $J$ is a continuous operator. Thus the continuous mapping theorem and $( X,\widetilde{U}^{m_k',n})\overset{d}{\Rightarrow}(\bar X,\bar{U}^{\infty,n})$ in $\mbf C([0,T];H)\times\mbf C([0,T];H_n)$ as $k\to\infty$ give that
 	$I_0^{m_k',n}=J(X,\widetilde{U}^{m_k',n})\overset{d}{\Rightarrow} J(\bar{X},\bar{U}^{\infty,n})$ in $\mbf C([0,T];H_n)$. Accordingly, one has
 	\begin{align*}
 		\bar{I}_0^{\infty,n}(t)=\int_0^tE_n(t-s)P_n\mcal D F(\bar{X}(s))\bar{U}^{\infty,n}(s)\ud s,~t\in[0,T].
 	\end{align*}
 	Further, by Lemma \ref{Imnconverge}, it holds that
 	$I_1^{m'_k,n}(t)\overset{d}{\Rightarrow}\frac{T}{2}\int_0^tE_n(t-s)P_n\mcal D F(X(s))A_nX(s)\ud s$, which together with 
 	$I_1^{m'_k,n}\overset{d}{\Rightarrow}\bar I_1^{\infty,n}$ in $\mbf C([0,T];H_n)$ and $X\overset{d}{=}\bar{X}$ in $\mbf C([0,T];H_n)$ yields
 	\begin{align}\label{sec3eq10}
 		\bar{I}_1^{\infty,n}(t)=\frac{T}{2}\int_0^tE_n(t-s)P_n\mcal D F(\bar X(s))A_n\bar X(s)\ud s,~t\in[0,T].
 	\end{align}
 	Similarly to the proof of \eqref{sec3eq10}, using  $(B^{Q_n},\widetilde{B}^{Q_n},\bar X)\overset{d}{=}(W^{Q_n},\widetilde{W}^{Q_n},X)$ in $\mbf C([0,T];\mcal H_n)^{\otimes 2}\times\mbf C([0,T];H)$, we obtain that for $t\in[0,T]$,
 	\begin{align*}
 		\bar I_2^{\infty,n}(t) &=\frac{T}{2} \int_0^tE_n(t-s)P_n\mcal D F(\bar X(s))P_nF(\bar X(s))\ud s,  \\
 		\bar I_3^{\infty,n}(t) &=\frac{T}{2}\int_0^tE_n(t-s) P_n\mcal D F(\bar X(s))P_n\ud  B^{Q_n}(s) +\frac{\sqrt{3}T}{6}\int_0^tE_n(t-s)P_n\mcal D F(\bar X(s))P_n\ud \widetilde{B}^{Q_n}(s), \\
 		\bar I_4^{\infty,n}(t) &=\frac{T}{4}\int_0^tE_n(t-s) \sum_{k=1}^{n}P_n\mcal D^2 F(\bar{X}(s))(P_nQ^{\frac{1}{2}}h_k, P_nQ^{\frac{1}{2}}h_k)\ud s.
 	\end{align*}
 	Plugging the expressions of $I_i^{\infty,n}$, $i=0,1,2,3,4$, into \eqref{sec3eq11} and using the fact 
 	$(B^{Q_n},\widetilde{B}^{Q_n},\bar X)\overset{d}{=}(W^{Q_n},\widetilde{W}^{Q_n},X)$ in $\mbf C([0,T];\mcal H_n)^{\otimes 2}\times\mbf C([0,T];H)$, we arrive at $\bar{U}^{\infty,n}(t)\overset{d}{=} \widetilde{U}^{\infty,n}(t)$ in $H_n$ due to the uniqueness of weak solution of \eqref{Utildeinftyn}, and thus have
 	$\widetilde{U}^{m'_k,n}(t)\overset{d}{\Rightarrow}\widetilde{U}^{\infty,n}(t)$ in $H_n$.  Since any subsequence of $\{\widetilde{U}^{m,n}(t)\}_{m\ge 1}$ contains further some  subsequence converging to $\widetilde{U}^{\infty,n}(t)$ in distribution in $H_n$,  $\widetilde{U}^{m,n}(t)\overset{d}{\Rightarrow}\widetilde{U}^{\infty,n}(t)$ in $H_n$ as $m\to\infty$. This finishes the proof.  \hfill $\square$

\subsection{Proof of Lemma \ref{Utildeinftyn-converge}}
Note that the $\|\cdot\|$-norm of the last term in \eqref{U} is finite almost surely. In fact, it follows from \eqref{semigroup1}, \eqref{F''bound} and \eqref{trQfinite} that
	\begin{align*}
		&\;\Big\|\int_{0}^{t}E(t-s)\sum_{k=1}^{\infty}\mcal D^2 F(X(s))(Q^{\frac{1}{2}}h_k, Q^{\frac{1}{2}}h_k)\ud s\Big\|\\
		\le&\; L\int_0^t\|(-A)^{-\frac{\eta}{2}}E(t-s)\|_{\mcal L(H)}\sum_{k=1}^\infty\|Q^{\frac{1}{2}}h_k\|^2\ud s\le K_{\eta,T}\|Q^{\frac{1}{2}}\|_{\mcal L_2(H)}^2.
	\end{align*}
	By a standard argument, we have $\|U(t)\|_{\mbf L^2(\Omega;H)}\le K_T$, $t\in[0,T]$. Further, using $E_n(\cdot)P_n=E(\cdot)P_n$, we decompose $\widetilde{U}^{\infty,n}(t)-U(t)$ into
\begin{align}\label{sec3eq12}
	\widetilde{U}^{\infty,n}(t)-U(t)=&\;\int_0^t E(t-s)P_n\mcal D F(X(s))(\widetilde{U}^{\infty,n}(s)-U(s))\ud s\notag\\
	&\;+T_0^n(t)-\frac{T}{2}\sum_{i=1}^{3}T_i^n(t)-\frac{\sqrt{3}T}{6}T_4^n(t)-\frac{T}{4}T_5^n(t),~t\in[0,T],
\end{align}
	where
	\begin{align*}
		T_0^n(t)=&\;\int_{0}^{t}E(t-s)(P_n-I)\mcal D F(X(s))U(s)\ud s, \\
		T_1^n(t)=&\;\int_0^t(P_n-I)E(t-s)\mcal D F(X(s))A_nX(s)\ud s+\int_0^tE(t-s)\mcal D F(X(s))(P_n-I)AX(s)\ud s, \\
		T_2^n(t)=&\;\int_0^t(P_n-I)E(t-s)\mcal D F(X(s))P_nF(X(s))\ud s+\int_0^tE(t-s)\mcal D F(X(s))(P_n-I)F(X(s))\ud s,\\
		T_3^n(t)=&\;\int_0^t(P_n-I)E(t-s)\mcal D F(X(s))P_n\ud W^{Q_n}(s)+\int_0^tE(t-s)\mcal D F(X(s))(P_n-I)\ud W^{Q_n}(s)\\
		&\;-\int_0^tE(t-s)\mcal D F(X(s))\ud W^{Q-Q_n}(s)\\
		=:&\;T_{3,1}^n(t)+T_{3,2}^n(t)+T_{3,3}^n(t),\\
		T_4^n(t)=&\;\int_0^t(P_n-I)E(t-s)\mcal D F(X(s))P_n\ud \widetilde W^{Q_n}(s)+\int_0^tE(t-s)\mcal D F(X(s))(P_n-I)\ud \widetilde W^{Q_n}(s)\\
		&\;-\int_0^tE(t-s)\mcal D F(X(s))\ud \widetilde W^{Q-Q_n}(s),
		\\	
		T_5^n(t)=&\; \int_0^t(P_n-I)E(t-s)\sum_{k=1}^n\mcal D^2 F(X(s))(P_nQ^{\frac{1}{2}}h_k,P_nQ^{\frac{1}{2}}h_k)\ud s\\
		&\;-\int_0^tE(t-s)\sum_{k=n+1}^\infty \mcal D^2 F(X(s))(P_nQ^{\frac{1}{2}}h_k,P_nQ^{\frac{1}{2}}h_k)\ud s\\
		&\;+\int_0^tE(t-s)\sum_{k=1}^\infty\big(\mcal D^2 F(X(s))(P_nQ^{\frac{1}{2}}h_k,P_nQ^{\frac{1}{2}}h_k)-\mcal D^2 F(X(s))(Q^{\frac{1}{2}}h_k,Q^{\frac{1}{2}}h_k)\big)\ud s\\
		=:&\;T_{5,1}^n(t)+T_{5,2}^n(t)+T_{5,3}^n(t).
	\end{align*}
	
	It follows from \eqref{semigroup1}, \eqref{Fgrow}, $\|(-A)^{-\rho}(P_n-I)\|_{\mcal L(H)}\le \lambda_{n+1}^{-\rho}$, $\rho\ge 0$ and $\|U(t)\|_{\mbf L^2(\Omega;H)}\le K$ that
	\begin{align*}
		\|T_0^n(t)\|_{\mbf L^2(\Omega;H)}&\le L\int_0^t\|(-A)^{-\rho}(P_n-I)\|_{\mcal L(H)}\|(-A)^\rho E(t-s)\|_{\mcal L(H)}\|U(t)\|_{\mbf L^2(\Omega;H)}\ud s\\
		&\le K\lambda_{n+1}^{-\rho},~\forall~\rho\in(0,1).
	\end{align*}
	
	By \eqref{semigroup1} and \eqref{F'}, for $\rho\in(0,1-\frac{\delta}{2})$, 
	\begin{align*}
		\|T_1^n(t)\|\le&\; L\int_0^t\|(-A)^{-\rho}(P_n-I)\|_{\mcal L(H)}\|(-A)^{\rho+\frac{\delta}{2}} E(t-s)\|_{\mcal L(H)}(1+\|X(s)\|_1)\|X(s)\|_1\ud s\\
		&\;+L\int_0^t\|(-A)^{\frac{\delta}{2}} E(t-s)\|_{\mcal L(H)}(1+\|X(s)\|_1)\|(-A)^{-\frac{\beta-1}{2}}(P_n-I)\|_{\mcal L(H)}\|X(s)\|_{\beta}\ud s\\
		\le &\; K\lambda_{n+1}^{-\rho}\int_0^t(t-s)^{-(\rho+\frac{\delta}{2})}(1+\|X(s)\|_1^2)\ud s+K\lambda_{n+1}^{-\frac{\beta-1}{2}}\int_0^t(t-s)^{-\frac{\delta}{2}}(1+\|X(s)\|_\beta^2)\ud s.
	\end{align*}
	Then Lemma \ref{Xregularity} yields
	\begin{align*}
		\sup_{t\in[0,T]}\|T_1^n(t)\|_{\mbf L^2(\Omega;H)}\le K_\rho(\lambda_{n+1}^{-\rho}+\lambda_{n+1}^{-\frac{\beta-1}{2}}),~\rho\in(0,1-\frac{\delta}{2}).
	\end{align*}
	Similarly, one has
	\begin{align*}
		\|T_2^n(t)\|\le&\; K\int_0^t\|(-A)^{-\frac{1}{2}}(P_n-I)\|_{\mcal L(H)}\|(-A)^{\frac{1}{2}} E(t-s)\|_{\mcal L(H)}(1+\|X(s)\|)\ud s\\
		&\;+K\int_0^t\|(-A)^{\frac{\delta}{2}} E(t-s)\|_{\mcal L(H)}(1+\|X(s)\|_1)\|(-A)^{-\frac{1}{2}}(P_n-I)\|_{\mcal L(H)}(1+\|X(s)\|)\ud s\\
		\le &\; K\lambda_{n+1}^{-\frac{1}{2}}\int_0^t(t-s)^{-\frac{1}{2}}(1+\|X(s)\|)\ud s+K\lambda_{n+1}^{-\frac{1}{2}}\int_0^t(t-s)^{-\frac{\delta}{2}}(1+\|X(s)\|_1^2)\ud s,
	\end{align*}
	which combined with Lemma \ref{Xregularity} leads to $\sup\limits_{t\in[0,T]}\|T_2^n(t)\|_{\mbf L^2(\Omega;H)}\le K \lambda_{n+1}^{-\frac{1}{2}}$.
	
	Applying the It\^o isometry, \eqref{semigroup1}, \eqref{Fgrow} and \eqref{trQfinite}, we get
	\begin{align*}
		\mbf E\|T_{3,1}^n(t)\|^2=&\;\int_0^t\|(P_n-I)E(t-s)\mcal D F(X(s))P_nQ_n^{\frac{1}{2}}\|_{\mcal L_2(H)}^2\ud s\\
		\le &\; K\int_0^t\|(-A)^{-\rho}(P_n-I)\|_{\mcal L(H)}^2\|(-A)^\rho E(t-s)\|_{\mcal L(H)}^2\|Q_n^{\frac{1}{2}}\|_{\mcal L_2(H)}^2\ud s
		\le  K\lambda_{n+1}^{-2\rho},~\rho\in(0,\frac{1}{2}),\\
		\mbf E\|T_{3,2}^n(t)\|^2=&\;\int_0^t\|E(t-s)\mcal D F(X(s))(P_n-I)Q_n^{\frac{1}{2}}\|_{\mcal L_2(H)}^2\ud s\\
		\le &\; K \|(-A)^{-\frac{\beta-1}{2}}(P_n-I)\|_{\mcal L(H)}^2\|(-A)^{\frac{\beta-1}{2}}Q_{n}^{\frac{1}{2}}\|_{\mcal L_2(H)}^2 \le K\lambda_{n+1}^{-(\beta-1)}\|(-A)^{\frac{\beta-1}{2}}Q^{\frac{1}{2}}\|_{\mcal L_2(H)}^2,\\
		\mbf E\|T_{3,3}^n(t)\|^2=&\;\int_0^t\|E(t-s)\mcal D F(X(s))(Q-Q_n)^{\frac{1}{2}}\|_{\mcal L_2(H)}^2\ud s\le K\sum_{k=n+1}^{\infty}\|Q^{\frac{1}{2}}h_k\|^2\to 0~\text{as}~n\to\infty.
	\end{align*}
	In this way, it holds that $\lim\limits_{n\to\infty}\sup\limits_{t\in[0,T]}\|T_3^n(t)\|_{\mbf L^2(\Omega;H)}=0$ and similarly $\lim\limits_{n\to\infty}\sup\limits_{t\in[0,T]}\|T_4^n(t)\|_{\mbf L^2(\Omega;H)}=0$.
	
	From \eqref{semigroup1}, \eqref{F''bound} and \eqref{trQfinite}, we derive that for any $\rho\in(0,1-\frac{\eta}{2})$,
	\begin{align*}
		\|T_{5,1}^n(t)\|\le K\int_0^t\|(-A)^{-\rho}(P_n-I)\|_{\mcal L(H)}\|(-A)^{\rho+\frac{\eta}{2}}E(t-s)\|_{\mcal L(H)}\sum_{k=1}^n\|Q^{\frac{1}{2}}h_k\|^2\ud s\le K\lambda_{n+1}^{-\rho}
	\end{align*} 
	and
	\begin{align*}
		\|T_{5,2}^n(t)\|\le K\int_0^t\|(-A)^{\frac{\eta}{2}}E(t-s)\|_{\mcal L(H)}\sum_{k=n+1}^\infty\|Q^{\frac{1}{2}}h_k\|^2\ud s\le K\sum_{k=n+1}^\infty\|Q^{\frac{1}{2}}h_k\|^2.
	\end{align*}
	Combining the bilinearity of $\mcal D^2 F(X(s))$, \eqref{semigroup1} and Assumption \ref{assum1},  we obtain
	\begin{align*}
		\|T_{5,3}^n(t)\|=&\;\Big\|\int_0^tE(t-s)\sum_{k=1}^\infty \mcal D^2 F(X(s))((P_n+I)Q^{\frac{1}{2}}h_k,(P_n-I)Q^{\frac{1}{2}}h_k)\ud s\Big\|\\
		\le &\;K\int_0^t\|(-A)^{\frac{\eta}{2}}E(t-s)\|_{\mcal L(H)}\sum_{k=1}^\infty \|Q^{\frac{1}{2}}h_k\|\|(P_n-I)Q^{\frac{1}{2}}h_k\|\ud s\\
		\le &\;K\int_0^t(t-s)^{-\frac{\eta}{2}}\|(-A)^{-\frac{\beta-1}{2}}(P_n-I)\|_{\mcal L(H)}\sum_{k=1}^\infty\|(-A)^{\frac{\beta-1}{2}}Q^{\frac{1}{2}}h_k\|^2\ud s\\
		\le&\; K\lambda_{n+1}^{-\frac{\beta-1}{2}}\|(-A)^{\frac{\beta-1}{2}}Q^{\frac{1}{2}}\|_{\mcal L_2(H)}^2.
	\end{align*}
	Consequently, $\lim\limits_{n\to\infty}\sup\limits_{t\in[0,T]}\|T_5^n(t)\|_{\mbf L^2(\Omega;H)}=0$.
	
	By \eqref{sec3eq12} and \eqref{Fgrow},
	\begin{align*}
		\|\widetilde{U}^{\infty,n}(t)-U(t)\|_{\mbf L^2(\Omega;H)}\le K\int_0^t 	\|\widetilde{U}^{\infty,n}(s)-U(s)\|_{\mbf L^2(\Omega;H)}\ud s+K\sum_{i=0}^5\sup_{t\in[0,T]}\|T_i^n(t)\|_{\mbf L^2(\Omega;H)},
	\end{align*}
	which yields $\lim\limits_{n\to\infty}\sup\limits_{t\in[0,T]}	\|\widetilde{U}^{\infty,n}(t)-U(t)\|_{\mbf L^2(\Omega;H)}=0$ due to the Gronwall inequality and previous estimates for $T_i^n$, $i=0,\ldots,5$. This implies the conclusion of Lemma \ref{Utildeinftyn-converge}. \hfill $\square$

\section{Applications of main result}
\label{Sec5}

In this section, we present two applications of Theorem \ref{maintheorem}.

\subsection{Asymptotic error distribution of a fully discrete AEE method}

Let $X^{m,n}$ be the fully discrete numerical solution based on the temporal AEE method and spatial spectral Galerkin method, i.e., $X^{m,n}$ is defined by
\begin{align*}
X^{m,n}(t)=E_n(t)P_nX_0+\int_0^tE_n(t-s)P_nF(X^{m,n}(\kappa_m(s)))\ud s+\int_0^tE_n(t-s)P_n\ud W^Q(s),~t\in[0,T]. 
\end{align*}
Then we immediately obtain the asymptotic error distribution of the above fully discrete method from Theorem \ref{maintheorem}.

\begin{cor}\label{Cor1}
Let Assumptions \ref{assum1} and \ref{assum2} hold. Moreover, assume that $\lambda_{n}\sim n^\alpha~(n\to\infty)$ for some $\alpha>0$. Then for any $\iota>\frac{2}{\alpha\beta}$ and $t\in[0,T]$, $m(X^{m,\lfloor m^\iota\rfloor}(t)-X(t))\overset{d}{\Rightarrow}U(t)$ in $H$ as $m\to\infty$.
\end{cor}

\begin{proof}
Using Lemmas \ref{WQ} and \ref{Xm-regularity}, one can show that for $\epsilon\ll 1$,
$$\sup_{t\in[0,T]}\sup_{m\ge 1}\|X^{m,n}(t)-X^{m}(t)\|_{\mbf L^2(\Omega;H)}\le K\big(\lambda_{n+1}^{-\frac{\beta}{2}}+\lambda_{n+1}^{-(1-\epsilon)}\big).$$
Then  $\lambda_{n}\sim n^\alpha~(n\to\infty)$ gives
\begin{align}\label{sec5eq1}
	\sup_{t\in[0,T]}m\|X^{m,\lfloor m^\iota\rfloor}(t)-X^{m}(t)\|_{\mbf L^2(\Omega;H)}\le Km^{1-\alpha\iota\min(\frac{\beta}{2},1-\epsilon)},\quad m\gg 1.
\end{align}
For both $\beta\in(1,2)$ and $\beta=2$, there is always $\epsilon\ll 1$ such that $\alpha\iota\min(\frac{\beta}{2},1-\epsilon)>1$ in view of $\iota>\frac{2}{\alpha\beta}$.  Thus, for any $t\in[0,T]$, $\|m(X^{m,\lfloor m^\iota\rfloor}(t)-X^{m}(t))\|$ converges to $0$ in probability, which combined with Theorem \ref{maintheorem} and Slutzky's theorem yields the desired conclusion. 
\end{proof}

\subsection{Asymptotic error distribution of  AEE method for SODE with additive noise}
Another direct application comes from   the finite-dimensional  version of Theorem \ref{maintheorem}.
Consider the following SODE
\begin{align}\label{SODE}
	\begin{cases}
	\ud Y(t)=C Y(t)\ud t+b(Y(t))\ud t+\ud W(t),~t\in[0,T] ,\\
	Y(0)=Y_0\in\mbb R^d,
	\end{cases}
\end{align} 
where $C\in\mbb R^{d\times d}$ is a negative definite matrix, $W$ is a $d$-dimensional standard Brownian motion, and $b:\mbb R^d\to \mbb R^d$ is Lipschitz continuous.  The AEE method for \eqref{SODE} reads 
\begin{align*}
	Y^m(t)=e^{Ct}Y_0+\int_0^te^{(t-s)C}b(Y^m(\kappa_m(s)))\ud s+\int_0^te^{(t-s)C}\ud W(s),~t\in[0,T].
\end{align*}
Note that \eqref{SODE} can be viewed a degenerated SPDE in the form of \eqref{SPDE} and $X^m$ is degenerated as $Y^m$, by letting $H=\mbb R^d$ and $Q\in\mcal L(\mbb R^d)$ be the identity matrix. Thus, as a direct application of Theorem \ref{maintheorem} in the finite-dimensional setting, we have the following result. 
\begin{cor}\label{Cor2}
	Let $C$ be a negative definite matrix and $b$ be Lipschitz  continuous.
If $b$ is three times continuously differentiable with $b^{(i)}$ being bounded for $i=1,2,3$, then $m(Y^m(t)-Y(t))\overset{d}{\Rightarrow} M(t)$ in $\mbb R^d$ as $m\to\infty$. Here, $M$ is the solution of the following SODE
\begin{align*}
	M(t)=&\;\int_0^te^{(t-s)C}b'(Y(s))M(s)\ud s-\frac{T}{2}\int_0^te^{(t-s)C}b'(Y(s))CY(s)\ud s\\
	&\;-\frac{T}{2}\int_0^te^{(t-s)C}b'(Y(s))b(Y(s))\ud s-\frac{T}{2}\int_0^te^{(t-s)C}b'(Y(s))\ud W(s)\\
	&\;-\frac{\sqrt{3}T}{6}\int_0^te^{(t-s)C}b'(Y(s))\ud \widetilde{W}(s)-\frac{T}{4}\int_0^te^{(t-s)C}\sum_{i=1}^{d}\frac{\PD^2}{\PD x_i^2}b(Y(s))\ud s,~t\in[0,T],
\end{align*}
where $\widetilde{W}$ is a $d$-dimensional standard Brownian motion independent of $W$.
\end{cor}

\begin{rem}
In Corollary \ref{Cor2}, we require $C$ to be a negative definite matrix. This corresponds to the setting in Section \ref{Sec2} where $A$ is a negative definite operator. We believe that Corollary \ref{Cor2} still holds, when $C$ is relaxed to any constant matrix.  This may be done by directly computing the limit distribution of $m(Y^m(t)-Y(t))$, instead of applying Theorem \ref{maintheorem}. Especially, we guess that when $C=0$, $M(t)$ is nothing but the asymptotic error distribution  of the Euler--Maruyama method for  \eqref{SODE}.
\end{rem}

\section{Concluding remarks}\label{Sec6}
In this paper, we establish the asymptotic error distribution of the temporal AEE method applied to  parabolic SPDEs with additive noise. To deal with the infinite-dimensional problem, we provide a uniform approximation theorem for convergence in distribution, and obtain the limit distribution of the normalized error process by studying that of its finite-dimensional approximation process. It turns out that the limit  process satisfies a linear SPDE. Our approach to deriving the asymptotic error distribution of numerical methods based on the approximation argument could potentially extend to other numerical methods for SPDEs driven by $Q$-Wiener processes. 
We finally conclude the paper with some future aspects of interest.
\begin{itemize}
\item In the case of SODEs, the normalized error process of many numerical methods can be shown to converge in distribution in $\mbf C([0,T];\mbb R^d)$. 
In the present paper, the asymptotic error distribution of the temporal AEE method is established for given $t\in[0,T]$. 
To further improve the result in $\mbf C([0,T];H)$ using Theorem \ref{uniform approximation} based on the current approximation $\widetilde{U}^{m,n}$, the result in Lemma \ref{Utildemn-Utildem} needs to be enhanced to
\[
\lim\limits_{n\to\infty}\sup_{m\ge 1}\mbf E\|\widetilde{U}^{m,n}-\widetilde{U}^m\|^2_{\mbf C([0,T];H)}=0,
\]
which is not a trivial extension.

\item The current result is given for the case of trace class noise with $\beta>1$ in Assumption \ref{assum1}. When $\beta\in(0,1]$, under the assumptions given in \cite{WangQi2015}, the temporal AEE method admits the mean-square order $\beta$. 
It implies that the normalized error process $U^m$ should be replaced by $m^{\beta}(X^m(t)-X(t))$, whose limit distribution is still unclear.
The main challenge through the current approach lies in the verification of Lemma \ref{euqiv} for $\beta\in(0,1]$.
 
 \item Corollary \ref{Cor1} gives the asymptotic error distribution of the fully discrete AEE method, where we set $n=\lfloor m^\iota\rfloor$ and require $\iota>\frac{2}{\alpha\beta}$ to ensure that the right-hand side of \eqref{sec5eq1} converges to $0$. When $\iota\le\frac{2}{\alpha\beta}$, Corollary \ref{Cor1} would be invalid and its study is more complicated due to the interaction of the spatial and temporal discretizations.
\end{itemize}

\bibliographystyle{plain}

\begin{thebibliography}{10}

\bibitem{CHJ2023}
C.~Chen, J.~Hong, and L.~Ji.
\newblock {\em Numerical {A}pproximations of {S}tochastic {M}axwell  {E}quations---via {S}tructure-{P}reserving {A}lgorithms},  {\em
  Lecture Notes in Math.} 2341.
\newblock Springer, Singapore, 2023.

\bibitem{Fukasawa2023}
M.~Fukasawa and T.~Ugai.
\newblock Limit distributions for the discretization error of stochastic
  {V}olterra equations with fractional kernel.
\newblock {\em Ann. Appl. Probab.}, 33(6B):5071--5110, 2023.

\bibitem{KolContinuity2018}
E.~Gobet and M.~Mrad.
\newblock Convergence rate of strong approximations of compound random maps,
  application to {SPDE}s.
\newblock {\em Discrete Contin. Dyn. Syst. Ser. B}, 23(10):4455--4476, 2018.

\bibitem{HLS2023}
J.~Hong, G.~Liang, and D.~Sheng.
\newblock Probabilistic evolution of the error of numerical method for linear
  stochastic differential equation.
\newblock {\em arXiv:2304.01602}.

\bibitem{HS2022}
J.~Hong and L.~Sun.
\newblock {\em Symplectic {I}ntegration of {S}tochastic {H}amiltonian
  {S}ystems}, {\em Lecture Notes in Math.} 2314.
\newblock Springer, Singapore, 2022.

\bibitem{HW2019}
J.~Hong and X.~Wang.
\newblock {\em Invariant {M}easures for {S}tochastic {N}onlinear
  {S}chr\"{o}dinger equations}, {\em Lecture Notes in
  Math.} 2251.
\newblock Springer, Singapore, 2019.

\bibitem{HuAAP}
Y.~Hu, Y.~Liu, and D.~Nualart.
\newblock Rate of convergence and asymptotic error distribution of {E}uler
  approximation schemes for fractional diffusions.
\newblock {\em Ann. Appl. Probab.}, 26(2):1147--1207, 2016.

\bibitem{Jacod1997}
J.~Jacod.
\newblock On continuous conditional {G}aussian martingales and stable
  convergence in law.
\newblock In {\em S\'{e}minaire de {P}robabilit\'{e}s, {XXXI}}, 
  {\em Lecture Notes in Math.} 1655, 232--246. Springer, Berlin, 1997.

\bibitem{Protter1998AOP}
J.~Jacod and P.~Protter.
\newblock Asymptotic error distributions for the {E}uler method for stochastic
  differential equations.
\newblock {\em Ann. Probab.}, 26(1):267--307, 1998.

\bibitem{Jentzen}
A.~Jentzen and P.~E. Kloeden.
\newblock Overcoming the order barrier in the numerical approximation of
  stochastic partial differential equations with additive space-time noise.
\newblock {\em Proc. R. Soc. Lond. Ser. A Math. Phys. Eng. Sci.},
  465(2102):649--667, 2009.

\bibitem{Klenke}
A.~Klenke.
\newblock {\em Probability {T}heory}.
\newblock Universitext. Springer, London, 2014.


\bibitem{Kloeden1992}
P.~E. Kloeden and E.~Platen.
\newblock {\em Numerical {S}olution of {S}tochastic {D}ifferential
  {E}quations},  {\em Appl. Math. (N. Y.)} 23.
\newblock Springer-Verlag, Berlin, 1992.

\bibitem{Kruse2012}
R.~Kruse and S.~Larsson.
\newblock Optimal regularity for semilinear stochastic partial differential
  equations with multiplicative noise.
\newblock {\em Electron. J. Probab.}, 17(65), 2012.

\bibitem{Protter1991}
T.~G. Kurtz and P.~Protter.
\newblock Wong-{Z}akai corrections, random evolutions, and simulation schemes
  for {SDE}s.
\newblock In {\em Stochastic analysis}, 331--346. Academic Press, Inc., Boston,
  MA, 1991.

\bibitem{Milstein2004}
G.~N. Milstein and M.~V. Tretyakov.
\newblock {\em Stochastic {N}umerics for {M}athematical {P}hysics}.
\newblock Scientific Computation. Springer-Verlag, Berlin, 2004.

\bibitem{AAtheorem}
J.~R. Munkres.
\newblock {\em Topology}.
\newblock Prentice Hall, Inc., Upper Saddle River, NJ, 2000.

\bibitem{Neuenkirch}
A.~Neuenkirch and I.~Nourdin.
\newblock Exact rate of convergence of some approximation schemes associated to
  {SDE}s driven by a fractional {B}rownian motion.
\newblock {\em J. Theoret. Probab.}, 20(4):871--899, 2007.

\bibitem{Nualart2023}
D.~Nualart and B.~Saikia.
\newblock Error distribution of the {E}uler approximation scheme for stochastic
  {V}olterra equations.
\newblock {\em J. Theoret. Probab.}, 36(3):1829--1876, 2023.

\bibitem{Protter2020SPA}
P.~Protter, L.~Qiu, and J.~S. Martin.
\newblock Asymptotic error distribution for the {E}uler scheme with locally
  {L}ipschitz coefficients.
\newblock {\em Stochastic Process. Appl.}, 130(4):2296--2311, 2020.

\bibitem{WangQi2015}
X.~Wang and R.~Qi.
\newblock A note on an accelerated exponential {E}uler method for parabolic
  {SPDE}s with additive noise.
\newblock {\em Appl. Math. Lett.}, 46:31--37, 2015.

\bibitem{Zhang2017}
Z.~Zhang and G.~E. Karniadakis.
\newblock {\em Numerical {M}ethods for {S}tochastic {P}artial {D}ifferential
  {E}quations with {W}hite {N}oise}, {\em Appl. Math. Sci.} 196.
\newblock Springer, Cham, 2017.

\bibitem{HuBIT}
H.~Zhou, Y.~Hu, and Y.~Liu.
\newblock Backward {E}uler method for stochastic differential equations with
  non-{L}ipschitz coefficients driven by fractional {B}rownian motion.
\newblock {\em BIT}, 63(3):Paper No. 40, 2023.

\end{thebibliography}

\end{document}